\documentclass[preprint,review,10pt]{elsarticle}

\usepackage{amsmath,amssymb,amsfonts,amsthm}
\usepackage{hyperref}
\usepackage{graphicx}
\usepackage{mathrsfs}
\usepackage{mathabx}

\usepackage{color}
 \usepackage{xcolor}
\usepackage{cases}

\usepackage{multirow}

 \usepackage{diagbox}  

\usepackage{extarrows}

\usepackage[all]{xy}

\usepackage{booktabs}  
\usepackage{caption}  

\newtheorem{assumption}{Assumption}

\newtheorem{corollary}{Corollary}
\newtheorem{lemma}{Lemma}
\newtheorem{remark}{Remark}
\newtheorem*{remark*}{Remark}
\newtheorem{theorem}{Theorem}

\let\b=\boldsymbol
\let\mr=\mathrm

\begin{document}

\begin{frontmatter}



\title{Pressure-robust  optimally convergent $H(\mr{div})$ finite element method without the commuting diagram property for the steady Oseen equations\tnoteref{funding} }

\tnotetext[funding]{
It is supported by National Natural Science Foundation of China (11771257) 
}

\author[label1] {Jin Zhang\corref{cor1}}
\author[label2] {Xiaowei Liu \fnref{cor2}}
\cortext[cor1] {Corresponding author: jinzhangalex@hotmail.com }
\fntext[cor2] {Email: xwliuvivi@hotmail.com }
\address[label1]{School of Mathematics and Statistics, Shandong Normal University,
Jinan 250014, China}
\address[label2]{School of Mathematics and Statistics, Qilu University of Technology (Shandong Academy of Sciences), Jinan 250353, China}


\begin{abstract}
This work develops a convergence theory for $H(\mr{div})$-conforming finite element methods applied to the steady Oseen problem, focusing on cases where the exact finite element complex holds while the commuting diagram property may fail. 
The proposed method incorporates vorticity stabilization to ensure optimal-order convergence of the velocity error, especially for convection-dominated cases.
As a crucial component of the analysis, exact de Rham and finite element complexes provide a framework whose utility includes establishing velocity error estimates independent of the discrete inf-sup constant.  As a representative example, Stenberg finite elements demonstrate the framework's  validity and offer  several computational advantages: pressure robustness, fewer degrees of freedom  than classical RT or BDM elements due to vertex continuity, and convergence   without requiring   the commuting diagram property.  Moreover, the proposed methodology is applicable to a class of finite element pairs that violate the commuting diagram property, thereby offering new possibilities for efficient discretizations of incompressible fluid problems, particularly in high Reynolds number regimes. 
\end{abstract}

\begin{keyword}
Oseen equations\sep H(div)-conforming finite elements \sep  Vorticity stabilization\sep  Pressure-robustness\sep  Stenberg elements.
\end{keyword}

\end{frontmatter}


%
%
%

 \section{Introduction}  

Classical inf-sup stable mixed finite element methods for incompressible flow problems typically relax the divergence constraint, leading to velocity errors that depend on pressure approximation. This results in a loss of pressure robustness \cite{MR3683678}, meaning the numerical scheme fails to preserve key invariance properties of the continuous problem. Such non-pressure-robust behavior can introduce significant inaccuracies, as discussed in  \cite{MR3683678} and \cite{MR3481034}. 
The development and analysis of pressure-robust schemes has been an active field of research, which encompasses the following directions: (1) constructing divergence-free finite element pairs, e.g., the Scott–Vogelius element \cite{MR813691,MR818790} on barycentric-refined meshes and the Guzmán–Neilan element \cite{MR3120580,MR3269433};  
(2) modifying the numerical scheme, such as Linke's approach \cite{MR3133522}; and  
(3) employing postprocessing strategies, as discussed in \cite[Remark 4.102]{MR3561143}.  
For the first popular strategy, \( H(\text{div}) \)-conforming methods have emerged as particularly attractive, since they enable the design of divergence-free finite element pairs on general meshes with greater flexibility than \( H^1 \)-conforming methods.  
For example, whereas \( H^1 \)-based Scott-Vogelius elements necessitate either specific mesh geometries or high polynomial degrees  on general triangulations \cite{MR813691,MR2772092,MR3882274}, \( H(\text{div}) \)-conforming methods provide enhanced flexibility and robustness.

Traditional \( H(\text{div}) \)-conforming finite elements, such as Raviart-Thomas (RT) elements \cite{MR483555,MR431752} and Brezzi-Douglas-Marini (BDM) elements \cite{MR859922,MR890035}, are widely used for incompressible flow simulations \cite{MR2860674,MR3712173,MR3780790,MR4410752,MR4200737,MR4098216} due to their satisfaction of the commuting diagram property and optimal approximation error estimates.  In contrast, elements that fail to satisfy the commuting diagram property often exhibit velocity errors inversely proportional to the discrete inf-sup constant—unlike those that satisfy the property. This dependency becomes critical when the inf-sup constant degenerates (e.g., under anisotropic mesh refinement), as it leads to unbounded error growth and loss of numerical stability.
Certain modern \( H(\text{div}) \)-conforming elements, such as Stenberg elements \cite{MR2594344} , while not satisfying the commuting diagram property, offer notable advantages: inf-sup stability, exactly divergence-free solutions, and a substantial reduction in degrees of freedom relative to RT or BDM elements, achieved through their vertex continuity constraints. Nevertheless, elements of Stenberg type have seen limited adoption in incompressible flow simulations, where robust control of velocity errors is essential, despite these advantageous features.
  The present study seeks to address this gap by establishing a convergence theory for \( H(\text{div}) \)-conforming methods in the context of the Oseen problem, with particular emphasis on scenarios where the exact finite element complex holds but the commuting diagram property fails.

A key challenge in \( H(\text{div}) \)-conforming methods is the lack of \( H^1 \)-conformity, necessitating discontinuous Galerkin (DG) techniques for the discretization of diffusive and convective terms.   Within the DG framework, upwinding is commonly employed to stabilize convection-dominated flows \cite{MR2882148}. For upwind-based  RT and BDM elements of order \( k \), the orthogonal properties of RT elements enable proof of \( k + 1/2 \)-order convergence  for the velocity field in the \( L^2 \)-norm \cite{MR4098216,MR4200737,MR4410752}. Analogous results have also been established for   $H^1$-conforming methods including Scott-Vogelius (SV) elements \cite{MR4331937}. Unfortunately, extending these convergence results to \( H(\text{div}) \)-conforming methods that violate the commuting diagram property, such as Stenberg elements, remains an open problem. 
To address this, we propose a stabilization approach that combines vorticity stabilization and upwinding, guaranteeing both numerical and theoretical 
\( k + 1/2 \)-order convergence while preserving pressure robustness. This stabilization strategy draws inspiration from existing DG and convection stabilization techniques but is specifically adapted to \( H(\text{div}) \)-conforming settings.

 This work focuses on the Oseen problems, encompassing but not limiting to convection-dominated cases. An \( H(\text{div}) \)-conforming finite element method, stabilized with upwinding and vorticity stabilization, is developed to solve the problems. The analysis establishes pressure-robust \( L^2 \)-norm convergence of order \( k + 1/2 \) for the velocity approximation, leveraging exact de Rham complexes and finite element complexes even when the commuting diagram property is violated. Beyond its application to Stenberg elements, the proposed methodology offers a general framework for analyzing a broader class of finite element pairs that lack the commuting diagram property, thereby expanding the range of efficient discretizations for fluid problems. Additionally, this work extends the concept of vorticity stabilization from \( H^1 \)-conforming methods to \( H(\text{div}) \)-conforming ones, opening new avenues for robust and accurate simulations of incompressible flows, particularly in high Reynolds number regimes.

We adopt standard notation for Sobolev and Lebesgue spaces. Given a domain \( D \subseteq \mathbb{R}^d \) (\( d = 2, 3 \)) and an exponent \( q \in [1, +\infty] \), the space \( L^q(D) \) consists of measurable functions with finite \( q \)-norm:  For \( q < +\infty \), \( L^q(D) \) contains functions whose \( q \)-th power is integrable over \( D \).  
 For \( q = +\infty \), it consists of essentially bounded functions.  The corresponding norm is denoted by \( \| \cdot \|_{0,q,D} \), except when \( q = 2 \), where we simplify the notation to \( \| \cdot \|_{0,D} \). In addition, the inner product in $L^2(D)$ will be denoted by $(\cdot , \cdot )_D$. The subspace \( L^q_0(D) \) includes functions in \( L^q(D) \) with zero mean.  For \( p \in [1, +\infty] \) and \( m \geq 0 \), the Sobolev space \( W^{m,p}(D) \) comprises \( L^p(D) \)-functions with weak derivatives up to order \( m \) also in \( L^p(D) \). The associated norm and seminorm are \( \| \cdot \|_{m,p,D} \) and \( | \cdot |_{m,p,D} \), respectively. When \( p = 2 \), we write \( H^m(D) = W^{m,2}(D) \), equipped with the norm \( \| \cdot \|_{m,D} \). The closure of \( C_0^\infty(D) \) in this norm is \( H_0^m(D) \).  
For brevity, we omit the domain subscript \( D \) in norms 
and inner products  when \( D = \Omega \).    Additionally, let $\mathbb{P}_k(K)$ denote the space of scalar-valued polynomials of degree at most $k$ on the element $K$, and $\mathbb{P}_k(K)^d$ its vector-valued counterpart in $\mathbb{R}^d$.

\section{Oseen equations and finite element methods}
Let $\Omega  \subset  \mathbb{R}^d$ ($d\in \{ 2,3 \} $) be a bounded, polyhedral, contractible domain  with  Lipschitz
continuous boundary $\partial\Omega$.  We consider the following
Oseen equations, for simplicity of presentation
with homogeneous Dirichlet data: find the velocity field $\b{u}$ and the pressure $p$ such that
\begin{equation}\label{eq:Oseen-equations}
\begin{aligned}
\mathcal{L}\b{u}+\nabla p=&\b{f}\quad \text{in $\Omega$},\\
\nabla \cdot \b{u}=&0\quad \text{in $\Omega$},\\
 \b{u}=&0\quad \text{on $\partial\Omega$},
\end{aligned}
\end{equation}
where $\mathcal{L}\b{u}:=-\nu \Delta \b{u}+(\b{b}\cdot \nabla )\b{u}+c \b{u}$, $\nu>0$ is the kinematic viscosity, $\b{b}$ is the advection velocity, $c$ is a given scalar function and $\b{f}$ is the external body force. 
The Oseen-type problem arises as a fundamental auxiliary formulation in numerous numerical schemes for solving the Navier–Stokes equations (see \cite[Remark 5.1]{MR3561143}).  In computational practice, this linearized problem appears in two primary contexts:  For steady-state simulations, the Picard iteration method generates a sequence of Oseen problems where the reaction coefficient vanishes $c \equiv 0$ and the convection term employs the velocity field from the preceding iteration as the advective velocity $\mathbf{b}$.  In time-dependent simulations utilizing semi-implicit temporal discretization, each time step requires solving an Oseen problem where  $\mathbf{u}$ represents the unknown velocity at the new time level, $\mathbf{b}$ corresponds to the explicitly treated advective velocity (typically extrapolated from previous time steps), and the reaction coefficient $c = 1/\Delta t$ emerges naturally from the time-stepping scheme. 
Focusing primarily (though not exclusively) on the high-Reynolds-number regime,   and motivated by  \cite[Remark 5.3]{MR3561143}, we impose the following condition:  
\begin{equation}\label{eq:condition-on-b+c}  
c(\mathbf{x}) - \frac{1}{2} \nabla \cdot \mathbf{b}(\mathbf{x})
\ge r_0
 > 0, \quad \mathbf{x} \in \Omega, 
\end{equation}  
where $r_0$  is a strictly positive constant.  

 Let $V:= H^1_0(\Omega)^d$
and $Q:=L^2_0(\Omega)$. Then the weak form of \eqref{eq:Oseen-equations} reads as follows: Find $ \b{u}\in V$ and $p\in Q$ such that
\begin{equation}\label{eq:Oseen-equations-weak-form}
\begin{aligned}
\nu \mathcal{D}(\b{u},\b{v})+\mathcal{C}(\b{u},\b{v})+\mathcal{R}(\b{u},\b{v})+\mathcal{P}(\b{v},p)
 &= \langle \mathbf{f}, \mathbf{v} \rangle_{V' \times V} \quad &&\forall \mathbf{v} \in V, \\
\mathcal{P}(\b{u},q) &= 0 \quad &&\forall q \in Q,
\end{aligned}
\end{equation}
where
\begin{equation*}
\begin{aligned}
\mathcal{D}(\b{u},\b{v})=&(\nabla \mathbf{u}, \nabla \mathbf{v}) ,\quad  &&\mathcal{C}(\b{u},\b{v})=( (\mathbf{b} \cdot \nabla) \mathbf{u},\mathbf{v}),\\
\mathcal{R}(\b{u},\b{v})=& (c\mathbf{u}, \mathbf{v} \big),\quad &&\mathcal{P}(\b{v},p)=- (\nabla \cdot \mathbf{v}, p).
\end{aligned}
\end{equation*}
For vector fields we have \((\mathbf{u}, \mathbf{v}) = \int_\Omega \mathbf{u} \cdot \mathbf{v} \,d\mathbf{x}\), while for matrix-valued functions \((\nabla \mathbf{u}, \nabla \mathbf{v}) = \int_\Omega \nabla \mathbf{u} : \nabla \mathbf{v} \,d\mathbf{x}\) denotes the Frobenius inner product. 
The term \(\langle \mathbf{f}, \mathbf{v} \rangle_{V' \times V}\) denotes the duality pairing between the dual space \(V'\) and \(V\) for \(\mathbf{f} \in V'\) and \(\mathbf{v} \in V\).  Under  condition  \eqref{eq:condition-on-b+c} along with appropriate regularity assumptions on $\b{b}$ and $c$,    the existence and uniqueness of a solution $(\mathbf{u}, p) \in H^1_0(\Omega)^d \times L^2_0(\Omega)$ of \eqref{eq:Oseen-equations-weak-form} is guaranteed, as proved in  \cite{Bof1Bre2For3:2013-Mixed,MR851383,MR3561143}.

Next, we introduce the weakly  divergence-free space
\begin{align*} 
V_{\mr{div}}:= \{ \b{v}  \in  V:\  (\nabla \cdot \b{v},q)=0\quad \forall q\in Q \}.
\end{align*}
 This space consists of \( H^1 \)-vector fields with zero divergence in the weak sense and homogeneous Dirichlet boundary conditions.  Our theoretical framework  is based on potential functions of weakly divergence-free velocity fields and their corresponding divergence-free finite element approximations.  This leads us to consider the following natural function spaces that characterize the kernel of the divergence operator:
\begin{equation*} 
\begin{aligned}
Z:=&\left\{  z\in H^1(\Omega):\;\mr{curl}z\in H^1_0(\Omega)^2\right\},\quad \text{if $\Omega\subset \mathbb{R}^2$},\\
Z:=&\left\{  \b{z}\in H^1(\Omega)^3:\;\mr{curl}\b{z}\in H^1_0(\Omega)^3 \right\},\quad \text{if $\Omega\subset \mathbb{R}^3$}.
\end{aligned}
\end{equation*}
This space consists of functions in \( H^1(\Omega) \) (or \( H^1(\Omega)^3 \) in 3D) whose curl lies in the subspace \( H^1_0(\Omega) \). 
The connection between these spaces and divergence-free fields is made precise by the following fundamental result.
\begin{theorem}\label{the:potential-of-u} 
Let $\Omega\subset \mathbb{R}^d$  ($d\in \{ 2,3 \} $) be a bounded, Lipschitz  domain. For any $\b{u}\in H^r(\Omega)^d$ with $r\ge 1$ satisfying $\mr{div}\b{u}=0$, there exists a potential $\b{z}$ (a scalar when $d=2$, or a vector when $d=3$)  with components in $H^{r+1}(\Omega)$ such that
\begin{equation}\label{eq:u-to-z}
\mr{curl} \b{z} = \b{u} \quad \text{in } \Omega.
\end{equation}
Moreover, the following stability estimate holds
\begin{equation*}
\|\b{z}\|_{r+1} \leq C\|\b{u}\|_{r}
\end{equation*}
where C > 0 is independent of $\b{u}$.
\end{theorem}
\begin{proof}
The existence of such $\b{z}$ follows from the generalized Bogovskii operator theory \cite[Theorem 4.9(b)]{MR2609313}, utilizing the vanishing cohomology of contractible domains.
\end{proof}

 \begin{remark}
Since $\nabla \cdot V\subset Q$,  the weakly divergence-free condition implies pointwise divergence-free almost everywhere in $\Omega$. Thus, we may equivalentlly characterize $V_{\mr{div}}$ as
\begin{align*} 
V_{\mr{div}}:= \{ \b{v}  \in  V:\; \nabla \cdot \b{v}=0\quad \forall \b{x}\in \Omega \}.
\end{align*}
Theorem \ref{the:potential-of-u} yields the fundamental relationship between the potential space and divergence-free space:
\begin{equation}\label{eq:Z-V-div}
\mathrm{curl}Z = V_{\mathrm{div}}.
\end{equation}
\end{remark}

\subsection{Finite element spaces and de Rham complexes}
Let $\{ \mathcal{T}_h \}_{h>0}$ be a family of shape-regular meshes of $\Omega$ consisting of  triangles ($d=2$) or tetrahedra ($d=3$). For each element $K \in \mathcal{T}_h$, we denote its diameter by $h_K$ and define the mesh size $h := \max_{K \in \mathcal{T}_h} h_K$.   The set of all facets in $\mathcal{T}_h$ is denoted by $\mathcal{F}_h$, which is partitioned into:  the subset $\mathcal{F}^i_h$ of interior facets  and the subset $\mathcal{F}^{\partial}_h$ of boundary facets $F\subset \partial\Omega$.  
For each facet $F \in \mathcal{F}_h$, we fix a unit normal vector $\mathbf{n}_F$, chosen arbitrarily for interior facets and taken as the outward unit normal $\mathbf{n}$ for boundary facets.  Let $\varphi$ be a piecewise continuous (scalar-, vector- or tensor-valued) function. For any  facet $F\in\mathcal{F}_h$ and $\b{x}\in F$,  we define: 
$$
[\varphi]_F(\b{x})\overset{\mr{def}}{=}
\left\{
\begin{aligned}
&\lim\limits_{t\to 0+0} (\varphi(\b{x}-t \b{n}_F)-\varphi(\b{x}+t \b{n}_F))
\quad &&\text{if $F\in \mathcal{F}^i_h$},\\
&\varphi (\b{x})\quad &&\text{if $F\in \mathcal{F}^{\partial}_h$},
\end{aligned}
\right.
$$
and
$$
\{\varphi \}_F(\b{x})\overset{\mr{def}}{=}
\left\{
\begin{aligned}
&\lim\limits_{t\to 0+0}\frac{1}{2} (\varphi(\b{x}-t \b{n}_F)+\varphi(\b{x}+t \b{n}_F))
\quad &&\text{if $F\in \mathcal{F}^i_h$},\\
&\varphi (\b{x})\quad &&\text{if $F\in \mathcal{F}^{\partial}_h$}.
\end{aligned}
\right.
$$
The operators act componentwise
for vector- and tensor-valued functions.  The facet subscript $F$ is often omitted when clear from context.  Let $\langle \cdot, \cdot \rangle_F$ denote the $L^2$-inner product on facet $F$, defined by
$$
\langle v, w \rangle_F := \int_F v w \, \mathrm{d}s, \quad \forall v,w \in L^2(F).
$$
Additionally, we define the following broken inner products:
\begin{equation}\label{eq:broken-inner-products}
(v,w)_h:=\sum_{K\in\mathcal{T}_h}(v,w)_K,
\quad
\langle v,w\rangle_{\mathcal{F}^i_h}:=\sum_{F\in\mathcal{F}^i_h} \langle v,w\rangle_F,\quad
\langle v,w\rangle_{\mathcal{F}_h}:=\sum_{F\in\mathcal{F}_h} \langle v,w\rangle_F
\end{equation}
with associated norms $\Vert \cdot \Vert_h$, $\Vert \cdot \Vert_{h,\mathcal{F}^i_h}$, $\Vert \cdot \Vert_{h,\mathcal{F}_h}$, respectively.

The finite element spaces employed in this work  are deeply connected to two classical de Rham complexes, which serve as a cornerstone for both theoretical analysis and the development of stable numerical methods for partial differential equations \cite{MR2594630,MR2269741,MR3908678}. For a two-dimensional domain $\Omega\subset \mathbb{R}^2$, the first complex \cite[\S 1.3 ]{MR851383} is given by
\begin{equation}\label{eq:deRham-L2-2D}
0 \hookrightarrow  H^1_0(\Omega) \xlongrightarrow{ \b{\mr{curl}} } 
\b{H}_0(\mr{div},\Omega) \xlongrightarrow{ \mr{div} }L^2_0(\Omega) \to 0,
\end{equation}
while for a three-dimensional domain $\Omega\subset \mathbb{R}^3$,  the second complex \cite[\S 4.2]{MR3802677} takes the form
\begin{equation}\label{eq:deRham-L2-3D}
0 \hookrightarrow   H^1_0(\Omega) \xlongrightarrow{ \b{\mr{grad}} }  
\b{H}_0(\mr{curl},\Omega)
\xlongrightarrow{ \b{\mr{curl}} } 
\b{H}_0(\mr{div},\Omega) \xlongrightarrow{ \mr{div} }L^2_0(\Omega) \to 0.
\end{equation}
Here, the function spaces are defined as
\begin{align*}
\mathbf{H}(\mathrm{div},\Omega) &:= \{\mathbf{v} \in [L^2(\Omega)]^d:\;\;  \mathrm{div}
\,\mathbf{v}\in L^2(\Omega) \}, \\
\mathbf{H}_0(\mathrm{div},\Omega) &:= \{\mathbf{v} \in \mathbf{H}(\mathrm{div},\Omega):\;\;\mathbf{v}|_{\partial\Omega}\cdot \b{n}=0 \}, \\
\mathbf{H}(\mathrm{curl},\Omega) &:= \{\mathbf{v} \in [L^2(\Omega)]^3: \;\; \mathbf{curl}\mathbf{v}\in [L^2(\Omega)]^3\},\\
\mathbf{H}_0(\mathrm{curl},\Omega) &:= \{\mathbf{v} \in \mathbf{H}(\mathrm{curl},\Omega) : \;\;  \mathbf{v}|_{\partial\Omega}\times \b{n}=\b{0}\},
\end{align*}
where $d = 2$ or 3 denotes the spatial dimension and $\b{n}$ is the unit normal vector of $\partial\Omega$.
 A sequence being a complex means that the composition of any two consecutive mappings vanishes. The complex is exact when the range of each map coincides precisely with the kernel of the subsequent one. By virtue of the domain $\Omega$ being contractible, the de Rham complexes \eqref{eq:deRham-L2-2D} (for $d=2$) and \eqref{eq:deRham-L2-3D} (for $d=3$) are both exact, as established in \cite{MR851383,MR2269741}.

We begin by defining the finite element space for velocity:
\begin{equation*} 
\begin{aligned}
V_h=\{&\b{v}_h\in \b{H}(\mr{div},\Omega): \;\b{v}_h|_K\in \b{V}_k(K),\;\forall K\in\mathcal{T}_h;\;  \; \b{v}_h\cdot \b{n}|_{\partial\Omega}=0 \}.
\end{aligned}
\end{equation*}
Here, the local space $\b{V}_k(K)$ is a set of vector-valued piecewise polynomials satisfying $ \mathbb{P}_k(K) ^d\subset \b{V}_k(K)$ and $ \mathbb{P}_{k+1}(K)^d\not\subset \b{V}_k(K)$. To maintain the generality of the theoretical framework presented in Section \ref{sec:convergence}, we do not impose further explicit constraints on the definition of $\b{V}_k(K)$.   
The finite element space \( V_h \) satisfies the conforming inclusion \( V_h \subset H_0(\text{div},\Omega) \), while typically failing to satisfy \( V_h \subset H_0^1(\Omega)^d \). This construction encompasses classical elements including RT and BDM elements, as well as more recent variants like Stenberg elements. The finite element space for pressure is denoted by $Q_h$ and satisfies $Q_h \subset L^2_0(\Omega)$.

Existing convergence analyses for incompressible flows employing $H(\mathrm{div})$-conforming elements—particularly RT and BDM elements—universally rely on the commuting diagram property. One important consequence of this property is that the velocity error estimates become independent of the discrete inf-sup constant. However, modern elements like the efficient Stenberg variant lack this property. We therefore develop a convergence theory based on exact finite element complexes, extending the framework to these elements. 
\begin{assumption}\label{assumption:Vh+Qh}
The pair  $V_h/Q_h$  forms  discrete finite element subcomplexes as follows:  
 \begin{equation}\label{eq:FE-subcomplex-2D}
\text{\textbf{2D case}}:\quad0 \hookrightarrow  Z_h
\xlongrightarrow{ \b{\mr{curl}} } 
V_h \xlongrightarrow{ \mr{div} }Q_h \to 0
\end{equation}
and
 \begin{equation}\label{eq:FE-subcomplex-3D}
\text{\textbf{3D case}}:\quad 0 \hookrightarrow  W_h \xlongrightarrow{ \b{\mr{grad}} }  
Z_h
\xlongrightarrow{ \b{\mr{curl}} } 
V_h \xlongrightarrow{ \mr{div} }Q_h \to 0. 
\end{equation}
These discrete complexes  satisfy the following fundamental properties:
\begin{enumerate}[(1)]
\item
They provide conforming discretizations of the continuous complexes \eqref{eq:deRham-L2-2D} and \eqref{eq:deRham-L2-3D}, respectively, with the following inclusions:
     \[
     \begin{aligned}
     Q_h &\subset L^2_0(\Omega),\\
     V_h &\subset \mathbf{H}_0(\mathrm{div};\Omega), 
     \end{aligned}
     \]
     and 
\begin{equation}\label{eq:regularity-Zh-0}
     \begin{cases}
     Z_h \subset H^1_0(\Omega) & \text{for } \Omega\subset\mathbb{R}^2, \\
     Z_h \subset \mathbf{H}_0(\mathrm{curl},\Omega), \ W_h \subset H^1_0(\Omega) & \text{for } \Omega\subset\mathbb{R}^3.
     \end{cases}
\end{equation}
\item
The discrete complexes are exact.
\item
The local approximation space $Z_h|_K$ on each element $K\in\mathcal{T}_h$ contains all  polynomial functions up to degree $\le k+1$:
\begin{equation}\label{eq:order-Zh}
\begin{aligned}
&\mathbb{P}_{k+1}(K) \subset Z_h|_K\quad \text{in the 2D case},\\
&\mathbb{P}_{k+1}(K)^3 \subset Z_h|_K\quad \text{in the 3D case}.
\end{aligned}
\end{equation}
\end{enumerate}

\end{assumption}

\begin{remark}
Based on the compatibility condition \eqref{eq:regularity-Zh-0}, we immediately derive the following results: For all facets  $F\in\mathcal{F}^i_h$, any $\b{z}\in Z$ and $ \b{z}_h\in Z_h$,  
 \begin{equation}\label{eq:regularity-Zh}
 \begin{aligned}
\left[\b{z}-\b{z}_h\right]|_F&=\b{0},\quad
 &&\text{if $\Omega\subset \mathbb{R}^2$},\\
\left[\b{z}-\b{z}_h\right]
\times \b{n}_F|_F&=\b{0},\quad 
&& \text{if $\Omega\subset \mathbb{R}^3$}.
\end{aligned}
\end{equation}
\end{remark}

\begin{remark}\label{rem:exactness-div}
The exactness $\mr{img}(\mr{div})=Q_h $ in  \eqref{eq:FE-subcomplex-2D} and \eqref{eq:FE-subcomplex-3D} implies the following conditions. 
\begin{itemize}
\item
$ \mr{div}V_h\supset Q_h $ guarantees the discrete inf-sup stability of
$V_h/Q_h$.
\item
$ \mr{div}V_h\subset Q_h $   ensures exactly divergence-free velocity approximations \cite{MR3683678}.
\end{itemize}

 The discretely  divergence-free space is defined by
\begin{align*} 
V_{h,\mr{div}}:= \{ \b{v}_h  \in  V_h:\; (\nabla \cdot \b{v}_h,q_h)=0\quad \forall q_h\in Q_h \}.
\end{align*}
The exactness property \(\mathrm{img}(\mathrm{curl}) = \mathrm{ker}(\mathrm{div})\) implies 
\begin{equation}\label{eq:Zh-Vh-div}
\mathrm{curl}\,Z_h = V_{h,\mathrm{div}}.
\end{equation}  
Combining \eqref{eq:Z-V-div} and \eqref{eq:Zh-Vh-div}, we observe that the approximation quality of \(V_{\mathrm{div}}\) by \(V_{h,\mathrm{div}}\) is inherently linked to the approximation of \(Z\) by \(Z_h\),  where the latter spaces exhibit simpler functional structures. This fundamental relationship plays a central role in our convergence analysis.

\end{remark}

The inclusion \eqref{eq:order-Zh} implies that $Z_h$ possesses optimal approximation properties.  Specifically, we have the following error estimate.
\begin{lemma}\label{lem:Z-Zh}
Assume  $\b{z}\in Z\cap H^{k+2}(\Omega)$. For every multi-index $\b{\alpha}=(\alpha_1,\ldots,\alpha_d)\in \mathbb{R}^d$,  the following estimation holds:
\begin{equation*}
\inf\limits_{\b{\psi}_h\in Z_h} \Vert \partial^{\alpha} (\b{z}-\b{\psi}_h) \Vert_h
\le C h^{k+2-|\alpha|} \vert \b{z} \vert_{k+2,\Omega}\quad \text{for $|\alpha|\le k+1$}.
\end{equation*}
\end{lemma}
 
 \begin{remark}
 For incompressible flows,  the commuting diagram property has been a standard requirement for $H(\mr{div})$-conforming elements in numerous classical works \cite{MR4098216,MR3780790,MR2860674,MR4410752}. 
  This property theoretically ensures that the velocity field error is independent of the discrete inf-sup stability constant.

Assumption \ref{assumption:Vh+Qh} replaces this classical requirement with the exactness of the discrete complex $(V_h, Q_h)$. Although exactness follows from the commuting diagram property, the reverse implication fails in general. This work thereby broadens the scope of viable discrete spaces while contributing new theoretical insights to mixed finite element analysis.
\end{remark}
 
 \subsection{Examples of finite element paris satisfying Assumption \ref{assumption:Vh+Qh}}
In fact, numerous finite element subcomplexes satisfy  Assumption \ref{assumption:Vh+Qh}; see, for instance,  \cite[Tables 1 \& 2]{MR3802677}, \cite[Example 5.10]{MR4654617} and discussions on the imposition of Dirichlet boundary conditions in  \cite{MR3045658,MR3802677}.
Specifically, we adopt the  \textbf{nonstandard $H(\mr{div})$ elements}   introduced by Stenberg \cite{MR2594344} for velocity approximation.  For the case where $\Omega\subset \mathbb{R}^2$ and $k\ge 2$, we consider the   subcomplex given in \cite[(5) and (12)]{MR3802677} or \cite[Example 4.6]{chen2023finite}:  
\begin{equation*} 
\begin{aligned}
 0 \hookrightarrow  Z_h:= \mr{Hermite}_{k+1} 
\xlongrightarrow{ \b{\mr{curl}} } 
 V_h:=\mr{Stenberg}_{k} \xlongrightarrow{ \mr{div} } 
 Q_h:=\mr{DG}_{k-1}
 \to 0,
\end{aligned}
\end{equation*}
where  
\begin{align*}
 \mr{Hermite}_{k+1}=&\{v\in C^0(\bar{\Omega}):\;v|_K\in \mathbb{P}_{k+1}(K),\;\forall K\in\mathcal{T}_h;\;v\in C^1(\mathcal{V}) ;\; v|_{\partial\Omega}=0\},\\
 \mr{Stenberg}_{k} =& \{\b{v}\in \b{H}_0(\mr{div},\Omega):\;\b{v}|_K\in  \mathbb{P}_k (K)^2,\;\forall K\in\mathcal{T}_h;\; \b{v}\in C^0(\mathcal{V}) \},\\
\mr{DG}_{k-1}=&\{v\in L^2_0(\Omega):\;\b{v}|_K\in \mathbb{P}_{k-1}(K),\;\forall K\in\mathcal{T}_h\}.
 \end{align*}
Here, $\mathcal{V}$ denotes the set of vertices in the triangulation $\mathcal{T}_h$, and $v\in C^r(\mathcal{V})$ represents that $v$ has continuous derivatives up to order $r$ at all vertices. The degrees of freedom (DoFs) can be given as follows:
\begin{itemize}
\item
For $v\in  \mr{Hermite}_{k+1} $:
\begin{itemize}
\item
function value $v(\b{x})$ and first order derivatives $\partial_i v(\b{x})$, $i=1,2$ at each vertex $\b{x}\in\mathcal{V}$,
\item
moments on each edge
$$
\int_F v\; q\mr{d}s,\quad q\in\mathbb{P}_{k+1-4}(F),\;\forall F\in\mathcal{F}_h,
$$
\item
moments on each element
$$
\int_K v\; q\mr{d}\b{x},\quad q\in\mathbb{P}_{k+1-3}(K),\;\forall K\in\mathcal{T}_h.
$$
\end{itemize}
\item
For $\b{v}\in  \mr{Stenberg}_k $:
\begin{itemize}
\item
function value $\b{v}(\b{x})$   at each vertex $\b{x}\in\mathcal{V}$,
\item
moments on each edge
$$
\int_F \b{v}\cdot \b{n}\; q\mr{d}s,\quad q\in\mathbb{P}_{k-2}(F),\;\forall F\in\mathcal{F}_h,
$$
\item
moments on each element
$$
\int_K \b{v}\cdot \b{q}\mr{d}\b{x},\quad \b{q}\in\mathcal{N}_{k-2}(K),\;\forall K\in\mathcal{T}_h,
$$
where $\mathcal{N}_{k-2}(K)$ is the Nédélec element of the first
kind of degree $k-2$.   
\end{itemize}
\item
For $v\in  \mr{DG}_{k-1} $:
\begin{itemize}
\item
moments on each element
$$
\int_K v\;q\mr{d}\b{x},\quad q\in\mathbb{P}_{k-1}(K),\;\forall K\in\mathcal{T}_h.
$$
\end{itemize}
\end{itemize}
The  implementation of homogeneous boundary conditions can be incorporated following the approach described in \cite[Section 4.1]{MR3802677}.

Similarly, for $\Omega\subset \mathbb{R}^3$ and  $k\ge 3$, we focus the   subcomplex in 
\cite{MR3802677} or \cite[Example 5.10]{MR4654617}: 
\begin{equation*} 
\begin{aligned}
&  0 \hookrightarrow
W_h:=\mr{Neilan}_{k+2} 
\xlongrightarrow{ \b{\mr{grad}} } 
Z_h:=\mr{HuZhang}_{k+1}  \\
&\xlongrightarrow{ \b{\mr{curl}} } 
V_h:=\mr{Stenberg}_{k} 
\xlongrightarrow{ \mr{div} } 
Q_h:= \mr{DG}_{k-1} 
 \to 0,
\end{aligned}
\end{equation*}
where the finite element spaces and their degrees of freedom (DOFs) are referred to \cite[Section 3.2]{MR3802677}.

\begin{remark}
Since $\Omega \subset \mathbb{R}^d$ is contractible, it follows from \cite{MR3802677} that the pair $V_h/Q_h = \mr{Stenberg}_k / \mr{DG}_{k-1}$ satisfies Assumption \ref{assumption:Vh+Qh} for $d=2$ with $k \ge 2$ as well as for $d=3$ with $k \ge 3$.
\end{remark}

\begin{remark}
The finite element pairs $V_h/Q_h = \mr{BDM}_k / \mr{DG}_{k-1}$ and $V_h/Q_h = \mr{RT}_k / \mr{DG}_{k}$ satisfy Assumption \ref{assumption:Vh+Qh} for both $d=2$ with $k\ge 1$  \cite[Example 4.6]{chen2023finite}  and $d=3$ with $k\ge 1$ \cite[(2.3.58)\&(2.3.62)]{Bof1Bre2For3:2013-Mixed}. Moreover, these two  families possess the commuting diagram property, which allows their convergence analysis to be directly established in the literature.
\end{remark}

 \subsection{Stabilized $H(\mr{div})$-conforming finite element method}
The  stabilized $H(\mr{div})$-conforming finite element method  of \eqref{eq:Oseen-equations} reads as follows:
Find $(\b{u}_h,p_h)\in V_h\times Q_h$ such that
\begin{equation}\label{eq:Hdiv-vorticity-stabilization}
\left\{
\begin{aligned}
A(\b{u}_h,\b{v}_h)+\mathcal{P}(\b{v}_h,p_h)=&\mathcal{G}(\b{v}_h)\quad \forall \b{v}_h\in V_h,\\
\mathcal{P}(\b{u}_h,q_h)=&0\quad \forall q_h\in Q_h,
\end{aligned}
\right.
\end{equation}
where $A(\b{u}_h,\b{v}_h):=\nu \mathcal{D}_h(\b{u}_h,\b{v}_h)+\mathcal{C}_h(\b{u}_h,\b{v}_h)+\mathcal{R}(\b{u}_h,\b{v}_h)+\mathcal{S}(\b{u}_h,\b{v}_h)$ and $
\mathcal{G}(\b{v}_h)=(\b{f},\b{v}_h)+\delta_0 (\tau \mr{curl} \b{f},  \mr{curl} \mathcal{L}\b{v}_h)_h$.  The components of $A(\b{u}_h,\b{v}_h)$ are presented in the remainder of this subsection.

\begin{remark}
From Remark \ref{rem:exactness-div},  one has $\mr{div}V_h\subset Q_h$  and  the second equation in \eqref{eq:Hdiv-vorticity-stabilization} implies
\begin{align*} 
V_{h,\mr{div}}:= \{ \b{v}_h  \in  V_h:\;  \nabla \cdot \b{v}_h=0\quad \forall \b{x}\in \Omega \}.
\end{align*}
Note that $V_h\not\subset V$, which implies 
$V_{h,\mr{div}}\not\subset V_{\mr{div}}$.

\end{remark}

In order to introduce $A(\cdot,\cdot)$, we first define the space
$$
V(h)=V_h+ \left[ V\cap H^3(\mathcal{T}_h) \right].
$$
For the discretisation of the diffusion term, we employ the standard symmetric interior penalty form $\mathcal{D}_h: \;V(h)\times V_h\to \mathbb{R}$  \cite{MR2882148}. The form $\mathcal{D}_h$  is defined
 by
 \begin{equation*} 
 \begin{aligned}
\mathcal{D}_h(\b{w},\b{v}_h):=& (\nabla \b{w},\nabla \b{v}_h)_h
-\sum_{F\in\mathcal{F}_h}\left(\langle \{ \nabla \b{w} \}\b{n}_F, [\b{v}_h]\rangle_F
+\langle  [\b{w}],\{ \nabla \b{v}_h \}\b{n}_F\rangle_F \right)\nonumber\\
&+\sum_{F\in\mathcal{F}_h}\frac{\sigma}{h_F}\langle  [\b{w}],[\b{v}_h]\rangle_F,
 \end{aligned}
 \end{equation*}
where  $[\nabla \b{w}]_{ij}=\frac{\partial w_i}{\partial x_j}$ and $h_F$ is the diameter of the face $F\in\mathcal{F}_h$. The penalty parameter $\sigma$ is a positive constant independent with $\nu$ and $h$ and has to be chosen sufficiently large to ensure the coercivity of $\mathcal{D}_h$.

 \begin{remark}\label{remark:DB}
The finite element method enforces the homogeneous Dirichlet boundary condition $\mathbf{u}|_{\partial\Omega} = \mathbf{0}$ through a hybrid strategy. Taking the Stenberg finite element as an example, the method combines:  (1) strong imposition in the discrete space $V_h$ by setting $\mathbf{u}_h(\mathbf{x}) = \mathbf{0}$  at all vertices $\mathbf{x} \in \partial\Omega$ and enforcing the normal flux condition $\int_F \mathbf{u}_h \cdot \mathbf{n} \, q \, \mathrm{d}s = 0$ for all $q \in \mathbb{P}_{k-1}(F)$ on each boundary face $F \in \mathcal{F}_h^{\partial}$;    (2) weak imposition  of the remaining boundary conditions through the numerical scheme.
 \end{remark}

 The reaction term $\mathcal{R}:\;V(h)\times V_h\to \mathbb{R}$  and the pressure-velocity coupling $\mathcal{P}:\; V(h)\times Q \to \mathbb{R}$ remains unchanged:
 \begin{equation*}
 \mathcal{R}(\b{u}_h,\b{v}_h)=(c \b{u}_h,\b{v}_h),\quad \mathcal{P}(\b{v}_h,q_h):=-(\nabla \cdot \b{v}_h,q_h)_h.
 \end{equation*}
 
 For the (linearised) inertia term, we introduce the convection term $\mathcal{C}_h:~ V(h)\times V_h\to \mathbb{R}$ with  upwinding and the vorticity stabilitzation $S:\;V(h)\times V_h\to \mathbb{R}$ as follows:
\begin{align}
\mathcal{C}_h(\b{u}_h,\b{v}_h):=&( (\b{b}\cdot \nabla)\b{u}_h,\b{v}_h )_h-\sum_{F\in\mathcal{F}_h^i}\langle  (\b{b}\cdot \b{n}_F)[\b{u}_h], \{ \b{v}_h \}\rangle_F\nonumber\\
&+\sum_{F\in\mathcal{F}_h}\gamma^c_F\langle  |\b{b}\cdot \b{n}_F|[\b{u}_h], [\b{v}_h ]\rangle_F,\nonumber\\
\mathcal{S}(\b{u}_h,\b{v}_h)=&\delta_0
\left\{ (\tau \mr{curl} \mathcal{L}\b{u}_h,  \mr{curl} \mathcal{L}\b{v}_h)_h+
\langle h^2[(\b{b}\cdot \nabla)\b{u}_h\times \b{n}], [(\b{b}\cdot \nabla)\b{v}_h\times \b{n}] \rangle_{ \mathcal{F}_h }
\right\}\nonumber,
\end{align}
where the broken scalar products are defined in \eqref{eq:broken-inner-products} and $\delta_0>0$ is a constant will be specified later. The weighting factor $\gamma^c_F$ is given by 
$$
\gamma^c_F=
\left\{
\begin{aligned}
&\frac{1}{2}\quad &&\text{ if $F\in\mathcal{F}^i_h$},\\
&\frac{1-\mr{sgn}(\b{b}\cdot \b{n})}{2}\quad  &&\text{if $F\in\mathcal{F}^{\partial}_h$},
\end{aligned}
\right.
$$
 and the stabilization parameter
$\tau|_K = \tau_K$ (also see \cite{MR4331937}) is defined as
\begin{equation*}
\tau_K:=\min\left\{1,\frac{\Vert \b{b} \Vert_{0,\infty} h_K }{\nu}  \right\} \frac{h^3_K}{\Vert \b{b} \Vert_{0,\infty}}.
\end{equation*}

\begin{remark}
 For $H^1$-conforming divergence-free inf-sup stable finite element pairs, the vorticity stabilization introduced in \cite{MR4331937} differs from classical SUPG methods \cite{MR571681} in its stabilization target. While standard SUPG stabilizes the convection-dominated term $(\mathbf{b}\cdot \nabla) \mathbf{u}$ directly, vorticity stabilization specifically acts on the divergence-free term $\mr{curl}((\mathbf{b}\cdot \nabla) \mathbf{u})$. 
Unlike conventional SUPG stabilization, which produces pressure-dependent velocity solutions through its inclusion of pressure terms in the stabilization operator, the vorticity stabilization, by design,  eliminates this coupling and rigorously maintains pressure-robust velocity approximations. 
\end{remark}

\begin{remark}\label{rem:convection-stabilization}
  The term $\mathcal{C}_h(\mathbf{u}_h,\mathbf{v}_h)$ represents the standard upwind discretization of the convection term.  In classical $H(\mr{div})$-conforming finite element methods, particularly BDM and RT elements, conventional upwinding schemes are consistently employed to stabilize the convection-dominated terms \cite{MR4098216,MR3780790,MR2860674,MR4410752}.

The vorticity stabilization $\mathcal{S}(\mathbf{u}_h,\mathbf{v}_h)$ is introduced to ensure that the velocity approximation error achieves an order of $k+1/2$ in the $L^2$ norm while preserving pressure robustness.  If $\mathcal{C}_h(\mathbf{u}_h,\mathbf{v}_h)$ is defined as  
\begin{equation}\label{eq:convection-central-flux}
\big( (\mathbf{b} \cdot \nabla)\mathbf{u}_h, \mathbf{v}_h \big)_h - \sum_{F \in \mathcal{F}_h^i} \big\langle (\mathbf{b} \cdot \mathbf{n}_F)[\mathbf{u}_h], \{\mathbf{v}_h\} \big\rangle_F+\sum_{F\in\mathcal{F}_h^{\partial}}\gamma^c_F\langle  |\b{b}\cdot \b{n}_F|[\b{u}_h], [\b{v}_h ]\rangle_F,  
\end{equation}
which corresponds to a central flux discretization, 
numerical experiments reveal certain instabilities.   
Previous studies \cite{MR1485996} have shown that in convection-dominated regimes, nonconforming finite element methods require not only SUPG-type stabilization but also additional jump penalization to maintain stability, as standalone SUPG stabilization may be insufficient.
\end{remark}

To conduct the analysis, we introduce a norm that is mesh-dependent, as follows:
\begin{equation}\label{eq:norm}
\begin{aligned}
\vvvert \b{v}_h \vvvert^2:=& 
 \Vert \nu^{1/2} \nabla \b{v}_h \Vert^2_{h}+\sum_{F\in\mathcal{F}_h}\sigma   \nu h_F^{-1}\langle  [\b{v}_h], [\b{v}_h ]\rangle_F 
\\
&+  \sum_{F\in\mathcal{F}_h}\langle  |\b{b}\cdot \b{n}_F|[\b{v}_h], [\b{v}_h ]\rangle_F
+ \mathcal{S}(\b{v}_h,\b{v}_h) \\
&+r_0\Vert  \b{v}_h \Vert^2_0.
\end{aligned}
\end{equation}

For the subsequent well-posedness and convergence analysis of numerical solutions, we require both continuous and discrete forms of the trace theorem. The following local trace inequality \cite{MR2249024} is fundamental: there exists a constant $C > 0$ such that for every element $K \in \mathcal{T}_h$, every facet $F \subset \partial K$, and every function $v \in H^1(K)$,
\begin{equation}\label{eq:trace-inequality}
| v |_{0,F} \leq C\left(h^{-1/2}_K | v |_{0,K} + h^{1/2}_K | v |_{1,K}\right).
\end{equation}

Additionally, we employ the discrete trace inequality (see \cite[Remark 1.47]{MR2882148}), which holds particularly because $\mathbf{V}_k(K)$ consists of vector-valued piecewise polynomials:  For any $\mathbf{v}_h \in V_h$ and any element $K \in \mathcal{T}_h$, the following trace inequality holds :
\begin{equation}\label{eq:discrete-trace-inequality}
| \mathbf{v}_h |_{0,\partial K} \leq C h^{-1/2}_K | \mathbf{v}_h |_{0,K}.
\end{equation}

\section{Wellposedness of the finite element method and convergence}\label{sec:convergence}
 To establish the well-posedness and convergence analysis for \eqref{eq:Hdiv-vorticity-stabilization},  we require two fundamental mathematical ingredients. First, we recall the following coercivity result:
\begin{lemma}\label{lem:coercivity}
 Suppose $\sigma>0$ is sufficiently large. Then, the bilinear form $A(\cdot,\cdot)$ is coercive on 
$V_h$ with respect to the energy norm $\vvvert\cdot \vvvert$, i.e.,
\begin{equation}\label{eq:coercivity}
A(\b{v}_h,\b{v}_h)\ge \frac{1}{2}\vvvert  \b{v}_h \vvvert^2\quad\forall \b{v}_h \in V_h.
\end{equation}
\end{lemma}
\begin{proof}
The proof follows from \cite[Lemma 4.59]{MR2882148}.
\end{proof}

Second,  the following discrete inf-sup condition for the finite element pair $V_h/Q_h$ is also needed. 
\begin{assumption}\label{assum:discrete-inf-sup-Vh-Qh}
There exists a constant $\beta^h_{is}>0$, independent of the mesh size $h$ and the viscosity $\nu$, such that
\begin{align}
&\inf\limits_{q_h\in Q_h}\sup\limits_{\b{v}_h\in V_h}\frac{\mathcal{P}( \b{v}_h,q_h)}{\Vert \b{v}_h \Vert_{1,h} \Vert q_h \Vert_{0}}=\beta_{is}^h,\label{eq:discrete-inf-sup-2}\\
&\inf\limits_{q_h\in Q_h}\sup\limits_{\b{v}_h\in V_h}\frac{\mathcal{P}( \b{v}_h,q_h)}{\vvvert \b{v}_h \vvvert \Vert q_h \Vert_{0}}\ge C (\nu+\Vert \b{b} \Vert_{0,\infty}  h+1)^{-1/2}\beta_{is}^h,\label{eq:discrete-inf-sup-1}
\end{align}
where  the mesh-dependent norm is defined as
$$
\Vert \b{v}_h \Vert_{1,h}^2= \Vert \nabla \b{v}_h  \Vert^2_h
+\sum_{F\in\mathcal{F}_h}h^{-1}_F \Vert [\b{v}_h ] \Vert^2_{0,F}.
$$
\end{assumption}
According to \cite[(1.3)]{MR1974504},  for all $\b{v}_h\in V_h$, we have the following stability estimate: 
 $$
 \begin{aligned}
 \Vert \b{v}_h \Vert^2_0\le & C \left[\Vert \nabla \b{v}_h \Vert^2_h
 +\sum_{F\in\mathcal{F}^i_h }   h^{-d}_F\left( \int_F [\b{v}_h]\mr{d}s\right)^2+\left(\int_{\partial\Omega} \b{v}_h \mr{d}s \right)^2\right]\\
 \le & C\Vert \b{v}_h \Vert_{1,h}^2.
 \end{aligned}
 $$ 
Combining this with inverse inequalities, we obtain the following bound 
\begin{align*}
 \vvvert \b{v}_h \vvvert^2\le & C(\nu+\Vert \b{b} \Vert_{0,\infty} h+ \max_{K\in\mathcal{T}_h} \tau_Kh^{-2}_K(\nu^2h^{-2}_K+\Vert \b{b} \Vert_{0,\infty}+1)  +1 ) \Vert \b{v}_h \Vert_{1,h}^2,\\
 \le &  C(\nu+\Vert \b{b} \Vert_{0,\infty} h+(\nu+h+h/\Vert \b{b} \Vert_{0,\infty})  +1 ) \Vert \b{v}_h \Vert_{1,h}^2
 \end{align*} 
and
\begin{equation}\label{eq:equivalence-triple-H1}
\begin{aligned}
&\vvvert \b{v}_h \vvvert\le C(\nu+\Vert \b{b} \Vert_{0,\infty}  h+1)^{1/2}\Vert \b{v}_h \Vert_{1,h},
\end{aligned}
\end{equation}
where $C$ is a positive constant independent of $h$ and $\nu$. 
Furthermore, if the discrete inf-sup condition \eqref{eq:discrete-inf-sup-2} holds, then the equivalence \eqref{eq:equivalence-triple-H1} immediately implies the validity of \eqref{eq:discrete-inf-sup-1}.

 \begin{lemma}\label{lem:stability-stenberg}
For the specific choice of finite element pair $V_h/Q_h=\mr{Stenberg}_{k} /\mr{DG}_{k-1} $, where  $k\ge 2$ in 2D or $k\ge 3$ in 3D,  the stability conditions in Assumption  \ref{assum:discrete-inf-sup-Vh-Qh} are satisfied.
\end{lemma}
\begin{proof}
The detailed proof can be found in the Appendix.
\end{proof}

The well-posedness of the discrete problem \eqref{eq:Hdiv-vorticity-stabilization} follows from two key ingredients: the coercivity property established in Lemma \ref{lem:coercivity} and the discrete inf-sup condition \eqref{eq:discrete-inf-sup-1}. A complete analysis of this mixed formulation can be found in   \cite{Bof1Bre2For3:2013-Mixed}.

Moreover, the numerical scheme \eqref{eq:Hdiv-vorticity-stabilization} maintains strong consistency when applied to sufficiently smooth velocity-pressure pairs $(\b{u},p)$, that is,
\begin{equation}\label{eq:consistency}
\left\{
\begin{aligned}
A(\b{u}-\b{u}_h,\b{v}_h)+\mathcal{P}(\b{v}_h,p-p_h)=&0\quad \forall \b{v}_h\in V_h,\\
\mathcal{P}(\b{u}-\b{u}_h,q_h)=&0\quad \forall q_h\in Q_h.
\end{aligned}
\right.
\end{equation}

\subsection{An error estimate for the velocity}
Let $\b{\psi}_h\in Z_h$ be arbitrary and set $\b{w}_h:=\mr{curl} \b{\psi}_h \in V_{h,\mr{div}}$.
Define 
$$
\b{u}-\b{u}_h =\b{\eta}-\b{\xi}_h.
$$
where
$\b{\eta}:= \b{u}-\b{w}_h=\mr{curl}(\b{z}-\b{\psi}_h)$ from Theorem \ref{the:potential-of-u} and $\b{\xi}_h:=\b{w}_h- \b{u}_h$.
\begin{remark}
Based on the exactness of the finite element subcomplex (Assumption \ref{assumption:Vh+Qh}), i.e., $\mr{curl}Z_h=V_{h,\mr{div}}$,  we establish a relationship between the approximation errors $\b{u}-\b{w}_h$ for   $\b{w}_h\in V_{h,\mr{div}}$ and $\b{z}-\b{\psi}_h$ for any $\b{\psi}_h\in Z_h$. The latter is more amenable to analysis due to the simpler structure of $Z_h$.
\end{remark}

Let $z$ be sufficiently regular. We define the norm:
\begin{equation*} 
\Vert \b{z} \Vert^2_{\ast}:=\vvvert \mr{curl} \b{z} \vvvert^2+(\nu+h)\sum_{s=0}^3h^{2s-4}
\Vert D^s \b{z} \Vert^2_h,
\end{equation*}
where  $D^s \b{z}$ denotes the derivative tensor  $(\partial^{\b{\alpha}}\b{z})_{|\b{\alpha}|=s}$, that is, gradient for $s = 1$, Hessian matrix for $s = 2$, etc.
Our analysis begins with a semi-robust estimate of the   convection  and reaction terms.
\begin{lemma}\label{lem:convection-bound}
There exists a constant $C>0$  independent of $h$ and $\nu$  such that
$$
|\mathcal{C}_h(\b{\eta},\b{v}_h)+\mathcal{R}(\b{\eta},\b{v}_h)|\le C   \Vert \b{z}-\b{\psi}_h \Vert_* \; \vvvert \b{v}_h \vvvert,\quad \forall  \b{v}_h \in  V_h.
$$
\end{lemma}
\begin{proof}
Integration by parts gives 
\begin{equation}\label{eq:ch}
\begin{aligned}
&\mathcal{C}_h(\b{\eta},\b{\xi}_h)+\mathcal{R}(\b{\eta},\b{\xi}_h)\\
=&\sum_{F\in\mathcal{F}_h}\int_{F}\b{b}\cdot \b{n}_F \{ \b{\eta} \}\cdot [ \b{\xi}_h ]\mr{d}s-( \b{\eta},(\b{b}\cdot \nabla)\b{\xi}_h )_h\\
&+\sum_{F\in\mathcal{F}_h}\gamma^c_F\langle  |\b{b}\cdot \b{n}_F|[\b{\eta}], [\b{\xi}_h ]\rangle_F+( (c-\nabla\cdot \b{b}) \b{\eta},\b{\xi}_h)\\
=:&\mathscr{C}_1+\mathscr{C}_2+\mathscr{C}_3+\mathscr{C}_4.
\end{aligned}
\end{equation}

We now bound each term on the right-hand side of \eqref{eq:ch}.
By the trace inequality \eqref{eq:trace-inequality}, for any facet $F\in\mathcal{F}_h$ 
\begin{equation}\label{eq:F-to-K}
\Vert [\b{\eta}] \Vert^2_{0,F}+\Vert \{ \b{\eta} \} \Vert^2_{0,F} \le 
C\sum_{K\in K_F}\left(  h^{-1} \Vert \b{\eta}\Vert^2_{0,K}+h \Vert \nabla \b{\eta} \Vert_{0,K}^2 \right)
\end{equation}
where $K_F:=\{ K\in\mathcal{T}_h:\, F\subset \partial K \}$.
Applying the Cauchy-Schwarz inequality and \eqref{eq:F-to-K}, we obtain
\begin{equation}\label{eq:C1+C3}
\begin{aligned}
|\mathscr{C}_1|+|\mathscr{C}_3|
 \le & 
 C\sum_{F\in\mathcal{F}_h} (\Vert [\b{\eta}] \Vert_{0,F}+\Vert \{ \b{\eta} \} \Vert_{0,F})\;
 \Vert\sqrt{ |\b{b}\cdot \b{n}_F| }[\b{\xi}_h ]\Vert_{0,F}\\
 \le & C \left( h^{-1} \Vert \b{\eta}\Vert^2_{0}+h \Vert \nabla \b{\eta} \Vert_{0}^2\right)^{1/2}\; \vvvert  \b{\xi}_h \vvvert
 \\
 \le &
 C\Vert \b{z}-\b{\psi}_h \Vert_*\; \vvvert  \b{\xi}_h \vvvert,
\end{aligned}
\end{equation}
and
\begin{equation}\label{eq:C4}
|\mathscr{C}_4|=( (c-\nabla\cdot \b{b}) \b{\eta},\b{\xi}_h)\le C \Vert \b{\eta} \Vert_0 \; \Vert \b{\xi}_h \Vert_0
\le C h^{1/2} \Vert \b{z}-\b{\psi}_h \Vert_*\; \vvvert  \b{\xi}_h \vvvert.
\end{equation}

We now proceed to estimate the most challenging term  $\mathcal{C}_2$. We reformulate $\mathcal{C}_2$ through integration by parts:
\begin{equation}\label{eq:C2-beta=1}
\begin{aligned}
\mathscr{C}_2=&-( \b{\eta},(\b{b}\cdot \nabla)\b{\xi}_h )_h
=-(\mr{curl} (\b{z}-\b{\psi}_h),(\b{b}\cdot \nabla)\b{\xi}_h )_h \\
=&-\sum_{F\in\mathcal{F}_h^i}\int_{F} [\b{z}-\b{\psi}_h]\cdot \{ (\b{b}\cdot \nabla)\b{\xi}_h \times \b{n}_F\} \mr{d}s \\
&-\sum_{F\in\mathcal{F}_h}\int_{F} \{ \b{z}-\b{\psi}_h \}\cdot  [ (\b{b}\cdot \nabla)\b{\xi}_h\times \b{n}_F ]\mr{d}s\\
&-(\b{z}-\b{\psi}_h ,  \mr{curl}((\b{b}\cdot \nabla)\b{\xi}_h))_h\\
=:&\mathscr{C}_{2,a}+\mathscr{C}_{2,b}+\mathscr{C}_{2,c}.
\end{aligned}
\end{equation}
When $\Omega\subset \mathbb{R}^3$, direct calculation yields
$$
[\b{z}-\b{\psi}_h]\cdot \{ (\b{b}\cdot \nabla)\b{\xi}_h \times \b{n}_F\} 
=-[\b{z}-\b{\psi}_h] \times \b{n}_F\cdot \{ (\b{b}\cdot \nabla)\b{\xi}_h\} \quad \forall F\in\mathcal{F}^i_h.
$$
Combining this identity with the regularity results from \eqref{eq:regularity-Zh}, we conclude that
\begin{equation}\label{eq:C2a}
\mathscr{C}_{2,a}=0.
\end{equation}
An application of the Cauchy-Schwarz inequality, together with the substitution $\eta\to \b{z}-\b{\psi}_h$  in \eqref{eq:F-to-K}, leads to  
\begin{equation}\label{eq:C2b}
\begin{aligned}
&|\mathscr{C}_{2,b}|=\left|\sum_{F\in\mathcal{F}_h}\int_{F} \{ \b{z}-\b{\psi}_h \}\cdot  [ (\b{b}\cdot \nabla)\b{\xi}_h\times \b{n}_F ]\mr{d}s\right|\\
\le & \left(\sum_{F\in\mathcal{F}_h}h^{-2}\Vert \{ \b{z}-\b{\psi}_h \} \Vert^2_{0,F} \right)^{1/2} 
\left(\sum_{F\in\mathcal{F}_h}h^2\Vert  [ (\b{b}\cdot \nabla)\b{\xi}_h\times \b{n}_F ] \Vert^2_{0,F} \right)^{1/2} \\
\le &
C h^{-1} \left( h^{-1}\Vert \b{z}-\b{\psi}_h \Vert^2_0+h\Vert \nabla (\b{z}-\b{\psi}_h) \Vert^2_h \right)^{1/2}\;  \vvvert  \b{\xi}_h \vvvert\\
 \le &
 C\Vert \b{z}-\b{\psi}_h \Vert_*\; \vvvert  \b{\xi}_h \vvvert.
\end{aligned}
\end{equation}
For every element $K\in\mathcal{T}_h$, the inverse inequality yields
\begin{align*}
\Vert  \mr{curl}(\nu \Delta \b{\xi}_h-c \b{\xi}_h) \Vert_{0,K}
\le & C h^{-1}_K(h^{-1}_K\nu \Vert \nabla  \b{\xi}_h\Vert_{0,K}+ \Vert c \Vert_{0,\infty,K} \Vert \b{\xi}_h  \Vert_{0,K})\\
\le &C (\nu^{1/2}h^{-2}+h^{-1})\vvvert  \b{\xi}_h \vvvert_K.
\end{align*}
Applying the Cauchy-Schwarz inequality, we obtain  
\begin{equation}\label{eq:C2c}
\begin{aligned}
&|\mathscr{C}_{2,c}|=|(\b{z}-\b{\psi}_h ,  \mr{curl}((\b{b}\cdot \nabla)\b{\xi}_h))_h|\\
=&(\b{z}-\b{\psi}_h ,  \mr{curl}(\mathcal{L}\b{\xi}_h))_h
+(\b{z}-\b{\psi}_h ,  \mr{curl}(\nu \Delta \b{\xi}_h-c \b{\xi}_h))_h\\
\le &
\left( \sum_{K\in\mathcal{T}_h} \tau_K^{-1}\Vert  \b{z}-\b{\psi}_h \Vert_{0,K}^2\right)^{1/2}
\left( \sum_{K\in\mathcal{T}_h} \tau_K\Vert  \mr{curl}(\mathcal{L}\b{\xi}_h) \Vert_{0,K}^2\right)^{1/2}\\
&+\sum_{K\in\mathcal{T}_h} \Vert  \b{z}-\b{\psi}_h \Vert_{0,K} \Vert  \mr{curl}(\nu \Delta \b{\xi}_h-c \b{\xi}_h) \Vert_{0,K}\\ 
\le &
C\left(h^{-3/2}+\nu^{1/2}h^{-2}\right) \Vert \b{z}-\b{\psi}_h \Vert_0\;  \vvvert \b{\xi}_h \vvvert\\
 \le &
 C\Vert \b{z}-\b{\psi}_h \Vert_*\; \vvvert  \b{\xi}_h \vvvert. 
\end{aligned}
\end{equation} 
 Substituting the bounds from \eqref{eq:C2a}---\eqref{eq:C2c} into \eqref{eq:C2-beta=1}, we derive the following estimate:
\begin{equation}\label{eq:C2-beta=1-bound}
|\mathscr{C}_2|\le C \Vert \b{z}-\b{\psi}_h \Vert_*\; \vvvert  \b{\xi}_h \vvvert.
\end{equation}
 Finally, combining the results from \eqref{eq:C1+C3}, \eqref{eq:C4} and   \eqref{eq:C2-beta=1-bound}, we obtain the desired estimate. 
\end{proof}

\begin{theorem}\label{the:uh-wh}
Let $\b{u}$ be the solution to \eqref{eq:Oseen-equations} and $\b{z}$ satisfy \eqref{eq:u-to-z}. 
Let $\b{u}_h\in V_h$ solve \eqref{eq:Hdiv-vorticity-stabilization}. Suppose Assumption \ref{assumption:Vh+Qh} holds.  Then, for any $\b{\psi}_h\in Z_h$ and  $\b{w}_h=\mr{curl} \b{\psi}_h\in V_{h,\mr{div}}$, we obtain
$$
\vvvert \b{u}-\b{u}_h \vvvert+\vvvert \b{u}_h-\b{w}_h \vvvert\le C \Vert \b{z}-\b{\psi}_h \Vert_*.
$$
\end{theorem}
\begin{proof}
From \eqref{eq:coercivity}, \eqref{eq:consistency} and the first equation in \eqref{eq:Hdiv-vorticity-stabilization}, we obtain 
\begin{equation}\label{eq:frame-for-uI-uh}
\begin{aligned}
\frac{1}{2}\vvvert \b{\xi}_h \vvvert^2\le& A(\b{\xi}_h,\b{\xi}_h)=A(\b{u}-\b{u}_h,\b{\xi}_h)-A(\b{\eta},\b{\xi}_h)\\
=&-\left( \nu \mathcal{D}_h(\b{\eta},\b{\xi}_h)+\mathcal{C}_h(\b{\eta},\b{\xi}_h)+\mathcal{R}(\b{\eta},\b{\xi}_h)+  \mathcal{S}(\b{\eta},\b{\xi}_h) \right).
\end{aligned}
\end{equation}
From \eqref{eq:consistency}, we have
$$
A(\b{u}-\b{u}_h,\b{\xi}_h)=-\mathcal{P}(\b{\xi}_h,p-p_h)=0,
$$
 since $\nabla \cdot \b{\xi}_h=0$, which follows from $\b{\xi}_h\in V_{h,\mr{div}}$.
 
Cauchy-Schwartz inequality and the trace inequality \eqref{eq:trace-inequality} for $\b{\eta}$ yield 
 \begin{equation*}
\begin{aligned}
&|\nu \mathcal{D}_h(\b{\eta},\b{\xi}_h)|\\
\le &  
\nu^{1/2}\Vert \nabla \b{\eta} \Vert\; \nu^{1/2}\Vert \nabla \b{\xi}_h \Vert
+\left(\sum_{F\in\mathcal{F}_h}\sigma^{-1} \nu  h_F \langle  \{ \nabla \b{\eta} \},\{ \nabla \b{\eta} \}\rangle_F\right)^{1/2} \left(\sum_{F\in\mathcal{F}_h}\sigma \nu  h_F^{-1}\langle  [\b{\xi}_h],[\b{\xi}_h]\rangle_F\right)^{1/2}\\
&+\left(\sum_{F\in\mathcal{F}_h} \nu  h_F^{-1} \langle  [ \b{\eta}], [\b{\eta}]\rangle_F\right)^{1/2} \left(\sum_{F\in\mathcal{F}_h}  \nu  h_F\langle \{ \nabla \b{\xi}_h \},\{ \nabla \b{\xi}_h \}\rangle_F\right)^{1/2}\\
&+\left(\sum_{F\in\mathcal{F}_h}\sigma \nu  h_F^{-1}\langle  [\b{\eta}],[\b{\eta}]\rangle_F\right)^{1/2} \left(\sum_{F\in\mathcal{F}_h}\sigma \nu  h_F^{-1}\langle  [\b{\xi}_h],[\b{\xi}_h]\rangle_F\right)^{1/2}\\
\le &
C\left(\nu  \sum_{i=0}^2h^{-2+2i}\Vert D^i \b{\eta} \Vert^2_h \right)^{1/2}\vvvert \b{\xi}_h \vvvert\\
\le &
 C \Vert \b{z}-\b{\psi}_h \Vert_*\vvvert \b{\xi}_h \vvvert.
\end{aligned}
\end{equation*}
From \eqref{eq:norm}, one has
\begin{align*}
|\mathcal{S}(\b{\eta},\b{\xi}_h)|
\le  \mathcal{S}(\b{\eta},\b{\eta})^{1/2} \mathcal{S}(\b{\xi}_h,\b{\xi}_h)^{1/2} 
\le  C \Vert \b{z}-\b{\psi}_h \Vert_* \vvvert \b{\xi}_h \vvvert.
\end{align*}
Combining the above estimations and Lemma \ref{lem:convection-bound} with \eqref{eq:frame-for-uI-uh}, we are done.
\end{proof}
Similar to the conclusion in \cite{MR4331937}, the accuracy of the solution's approximation depends solely on how well the space $Z_h$ approximates the space 
$Z$.  Furthermore, we obtain the following result from Lemma \ref{lem:Z-Zh}.  
\begin{corollary}\label{the:u-uh}
In addition  to the assumptions of Theorem \ref{the:uh-wh}, suppose that
$\b{u}\in H^1_0(\Omega)^d\cap H^{k+1}(\Omega)^d$. Then, there exists a constant $C>0$, independent of $h$ and $\nu$ , such that
$$
\vvvert \b{u}-\b{u}_h \vvvert\le C   h^k(h^{1/2}+\nu^{1/2})\Vert \b{u} \Vert_{k+1}.
$$
\end{corollary}

\begin{remark}
The velocity error is pressure-robust, independent of pressure approximation. In the construction of pressure-robust finite elements, $H(\mr{div})$-conforming elements--- particularly  RT  and  BDM  elements---have been widely adopted (see \cite{Bof1Bre2For3:2013-Mixed,MR2657851,MR2860674,MR3780790,MR4098216,MR4410752}). The popularity of these elements stems largely from their satisfaction of the commuting diagram property, which facilitates straightforward derivation of pressure-robust optimal error estimates for the velocity field. In contrast, another $H(\mr{div})$--conforming element---the Stenberg element---though featuring fewer degrees of freedom, has received comparatively less attention due to its failure to satisfy the commuting diagram property. Leveraging the property that Stenberg elements form an exact finite element subcomplex, this work establishes for the first time pressure-robust optimal velocity error estimates, thereby paving the way for further investigation of these elements. 
\end{remark}

\begin{remark}\label{rem:k+1/2-discussion}
When $\nu\le Ch$,  the velocity error attains the optimal convergence order of $k + \frac{1}{2}$ in convection-dominated regimes. This result is consistent with the sharpness bounds established in prior works: for RT and BDM elements in \cite{MR4098216}, and for Scott-Vogelius element in \cite{MR4331937}. For $H(\mr{div})$-conforming finite element methods of convection-dominated Oseen problems, particularly when using RT  and BDM elements,  the $k+1/2$-order convergence of the velocity field using $k$-th order elements relies crucially on three key factors: the commuting diagram property,  the special relationship between BDM and RT elements, and the orthogonal properties of the RT interpolation operator. In contrast, for other $H(\mr{div})$ elements such as Stenberg elements which generally lack these   properties, the analogous analytical techniques become inapplicable for achieving such convergence rates. In this work, we establish the $k+\frac{1}{2}$-order convergence through two key components: exact finite element subcomplexes and  vorticity-based stabilization. This approach is not limited to specific element choices and can be naturally extended to a wide range of finite element pairs.  
\end{remark}

\subsection{An error estimate for the pressure}
Denote by $\pi_h : L^2_0(\Omega ) \rightarrow  Q_h$ the $L^2(\Omega )$ orthogonal projection onto $Q_h$.
\begin{theorem}\label{the:p-error} 
Let $(\b{u},p)\in \b{H}^1_0(\Omega)^d\times L^2_0(\Omega)$ solve \eqref{eq:Oseen-equations} and
 $(\b{u}_h,p_h)\in V_h\times Q_h$ solve \eqref{eq:Hdiv-vorticity-stabilization}. Let   $\b{z}$ satisfy \eqref{eq:u-to-z}.  Suppose Assumption \ref{assumption:Vh+Qh}  and \ref{assum:discrete-inf-sup-Vh-Qh} hold. Then the following error estimate holds:  For any  $\b{\psi}_h\in Z_h$ and $\b{w}_h=\mr{curl} \b{\psi}_h\in V_{h,\mr{div}}$,
$$
\Vert \pi_h p-p_h \Vert_0\le C  \left( \vvvert \b{u}-\b{u}_h\vvvert+\vvvert \b{u}_h-  \b{w}_h\vvvert
+(h^{1/2}+\nu h^{-1/2}) \Vert \b{z}- \b{\psi}_h \Vert_*
\right).
$$
\end{theorem}
\begin{proof}
From \eqref{eq:discrete-inf-sup-2}, for $\pi_h p-p_h\in Q_h$, we can find 
$\b{v}^*_h\in V_h$ satisfying
\begin{equation}\label{eq:vh*}
\begin{aligned}
&\nabla \cdot \b{v}^*_h=\pi_h p-p_h,\quad \text{since $\mr{div}V_h\subset Q_h$},\\
&\Vert \b{v}^*_h \Vert_{1,h}\le  (\beta^h_{is})^{-1} \Vert \pi_h p-p_h \Vert_0.
\end{aligned}
\end{equation}
Combining this with the Poincar\'{e}-Friedrichs inequality \cite[(1.3)]{MR1974504}, we derive 
\begin{equation}\label{eq:function}
\Vert \b{v}^*_h \Vert^2_0\le C\left( \Vert \nabla \b{v}^*_h \Vert^2_h +\sum_{F\in\mathcal{F}_h}h^{-1}_F \Vert [\b{v}^*_h] \Vert^2_F\right)
\le 
C  \Vert \pi_h p-p_h \Vert^2_0.
\end{equation}

Let $\b{e}:=\b{u}-\b{u}_h$.  Then from \eqref{eq:vh*} and \eqref{eq:consistency}, one has
\begin{equation}\label{eq:pressure-0}
\begin{aligned}
\Vert \pi_h p-p_h \Vert^2_0=&(\mr{div} \b{v}^*_h, \pi_h p-p_h)
=\underbrace{ (\nabla \cdot \b{v}^*_h,\pi_h p-p) }_{=0,\; \text{since $\mr{div}V_h\subset Q_h$  }}-A(\b{e},\b{v}^*_h)\\
=&-( \nu \mathcal{D}_h(\b{e},\b{v}^*_h)+\mathcal{C}_h(\b{e},\b{v}^*_h)+\mathcal{R}(\b{e},\b{v}^*_h)+  \mathcal{S}(\b{e},\b{v}^*_h) ).
\end{aligned}
\end{equation}
We now analyze each term on the right-hand side in detail.

First, applying the Cauchy-Schwarz inequality along with \eqref{eq:vh*} gives
\begin{align*}
 (\nabla \b{e},\nabla \b{v}^*_h)_h\le &\Vert \nabla \b{e} \Vert_h  \; \Vert \nabla \b{v}^*_h \Vert_h 
 \le 
 C \nu^{-1/2} \vvvert \b{e}\vvvert \; \Vert \pi_h p-p_h \Vert_0,
 \end{align*}
 and
\begin{align*}
\sum_{F\in\mathcal{F}_h}\frac{\sigma}{h_F}\langle  [\b{e}],[\b{v}^*_h]\rangle_F
\le &  \left(\sum_{F\in\mathcal{F}_h}\sigma^2   h_F^{-1}\langle  [\b{e}],[\b{e}]\rangle_F\right)^{1/2} \left(\sum_{F\in\mathcal{F}_h}  h_F^{-1}\langle  [\b{v}^*_h],[\b{v}^*_h]\rangle_F\right)^{1/2} \\
 \le &
 C\nu^{-1/2} \vvvert \b{e}\vvvert \Vert \pi_h p-p_h \Vert_0.
 \end{align*}
Furthermore,  for any $\b{\psi}_h\in Z_h$ with $\b{w}_h=\mr{curl} \b{\psi}_h$,  the Cauchy-Schwarz inequality, trace inequality \eqref{eq:trace-inequality}, triangle inequality and inverse inequality  imply
 \begin{align*}
&\sum_{F\in\mathcal{F}_h}\langle \{ \nabla \b{e} \}\b{n}_F, [\b{v}^*_h]\rangle_F
\le  C\left( \sum_{F\in\mathcal{F}_h} h_F \langle  \{ \nabla \b{e} \},\{ \nabla \b{e} \}\rangle_F\right)^{1/2}  \left(\sum_{F\in\mathcal{F}_h}   h_F^{-1}\langle  [\b{v}^*_h],[\b{v}^*_h]\rangle_F\right)^{1/2}\\
\le & C\left( \sum_{F\in\mathcal{F}_h}  h_F   \sum_{K\in K_F} \left( h^{-1}_K \Vert \nabla \b{e} \Vert_{0,K}^2
+h_K \Vert D^2 \b{e} \Vert_{0,K}^2\right)  \right)^{1/2}  \Vert \pi_h p-p_h \Vert_0\\
\le & C\left(\Vert \nabla \b{e} \Vert^2_h +h^2 \Vert D^2 (\b{u}-\mr{curl} \b{\psi}_h) \Vert^2_h + h^2 \Vert D^2 (\b{u}_h-\b{w}_h) \Vert^2_h   \right)^{1/2}\Vert \pi_h p-p_h \Vert_0\\
\le & C\left(\Vert \nabla\b{e} \Vert^2_h +h^{-1} \Vert \b{z}- \b{\psi}_h \Vert^2_*+\Vert \nabla (\b{u}_h-  \b{w}_h) \Vert^2_h  \right)^{1/2}\Vert \pi_h p-p_h \Vert_0\\
\le & C\left(\nu^{-1/2} \vvvert \b{e}\vvvert
+h^{-1/2} \Vert \b{z}- \b{\psi}_h \Vert_*+\nu^{-1/2} \vvvert \b{u}_h-  \b{w}_h\vvvert\right) \Vert \pi_h p-p_h \Vert_0,
 \end{align*}
where $K_F$ has been defined in \eqref{eq:F-to-K}.  Using the Cauchy-Schwarz inequality and the discrete trace inequality \eqref{eq:discrete-trace-inequality}, we obtain
 \begin{align*}
\sum_{F\in\mathcal{F}_h}\langle  [\b{e}],\{ \nabla \b{v}^*_h \}\b{n}_F\rangle_F \le &
C\left(\sum_{F\in\mathcal{F}_h} h_F^{-1} \langle  [ \b{e}], [\b{e}]\rangle_F\right)^{1/2} \left(\sum_{F\in\mathcal{F}_h}    h_F\langle \{ \nabla \b{v}^*_h \},\{ \nabla \b{v}^*_h \}\rangle_F\right)^{1/2}\\
\le & 
C\nu^{-1/2} \vvvert \b{e}\vvvert \Vert  \nabla \b{v}^*_h \Vert_h 
\le 
C\nu^{-1/2} \vvvert \b{e}\vvvert \Vert  \pi_h p-p_h  \Vert_0.
\end{align*}
Thus, we conclude
\begin{equation}\label{eq:pressure-1}
\nu |\mathcal{D}_h( \b{e},\b{v}^*_h)|\le C  (\nu^{1/2} \vvvert \b{e}\vvvert+\nu^{1/2} \vvvert \b{u}_h-  \b{w}_h\vvvert
+\nu h^{-1/2} \Vert \b{z}- \b{\psi}_h \Vert_*)\; \Vert \pi_h p-p_h \Vert_0.
\end{equation}

Integration by parts leads to
\begin{align*}
&\mathcal{C}_h(\b{e},\b{v}^*_h)+\mathcal{R}(\b{e},\b{v}^*_h)\\ 
=&\sum_{F\in\mathcal{F}_h}\int_{F}\b{b}\cdot \b{n}_F \{ \b{e} \}\cdot [ \b{v}^*_h ]\mr{d}s-( \b{e},(\b{b}\cdot \nabla)\b{v}^*_h )_h\\
&+\sum_{F\in\mathcal{F}_h}\gamma^c_F\langle  |\b{b}\cdot \b{n}_F|[\b{e}], [\b{v}^*_h ]\rangle_F+( (c-\nabla\cdot \b{b}) \b{e},\b{v}^*_h)\\
=:&\mathscr{P}_1+\mathscr{P}_2+\mathscr{P}_3+\mathscr{P}_4.
\end{align*}
We now estimate each term individually. 
By the trace inequality \eqref{eq:trace-inequality}, triangle inequality  and inverse inequality, we have
\begin{align*}
&\sum_{F\in\mathcal{F}_h}\langle \{\b{e}\}, \{\b{e}\}\rangle_F
\le  C \sum_{ K\in\mathcal{T}_h }( h^{-1}_K\Vert \b{e} \Vert^2_{0,K}+h_K\Vert \nabla \b{e} \Vert^2_{0,K})\\
 \le &
 C(h^{-1}\Vert \b{e} \Vert^2_0+h \Vert \nabla (\b{u}-\mr{curl} \b{\psi}_h) \Vert^2_h+
h\Vert \nabla (\b{w}_h-\b{u}_h)\Vert^2_h
 )\\
 \le &
 C \left(h^{-1} \Vert \b{e} \Vert^2_0+ \Vert \b{z}- \b{\psi}_h \Vert_*^2+
h^{-1}\Vert \b{w}_h-\b{u}_h\Vert^2_0 \right).
\end{align*}
Then Cauchy-Schwartz inequality and  stability estimate \eqref{eq:vh*} yield
\begin{align*}
|\mathscr{P}_1|\le & 
Ch^{1/2}\left(\sum_{F\in\mathcal{F}_h}\langle \{\b{e}\}, \{\b{e}\}\rangle_F \right)^{1/2}
 \left(\sum_{F\in\mathcal{F}_h}h^{-1}_F\langle  [\b{v}^*_h], [\b{v}^*_h ]\rangle_F \right)^{1/2}\\ 
 \le &
 C(\vvvert \b{e} \vvvert^2+h\Vert \b{z}- \b{\psi}_h \Vert^2_*+
\vvvert \b{u}_h- \b{w}_h \vvvert^2 )^{1/2}   \Vert \pi_h p-p_h \Vert_0.
\end{align*}
From H\"{o}lder's inequality  and the stability estimate \eqref{eq:vh*}, we deduce 
\begin{align*}
|\mathscr{P}_3|
\le &
Ch^{1/2} \left(\sum_{F\in\mathcal{F}_h}\gamma^c_F \langle  |\b{b}\cdot \b{n}_F|[\b{e}], [\b{e} ]\rangle_F \right)^{1/2}
 \left(\sum_{F\in\mathcal{F}_h}h^{-1}_F\langle  |[\b{v}^*_h], [\b{v}^*_h ]\rangle_F \right)^{1/2}\\
 \le &
Ch^{1/2} \vvvert \b{e} \vvvert \Vert \pi_h p-p_h \Vert_0.
\end{align*}
Additionally, the Cauchy-Schwarz inequality, along with stability estimates \eqref{eq:vh*} and \eqref{eq:function},   gives 
\begin{align*}
|\mathscr{P}_2|+|\mathscr{P}_4|
\le &
C\Vert \b{e} \Vert_0 \Vert \nabla \b{v}^*_h \Vert_h+C \Vert \b{e} \Vert_0 \Vert \b{v}^*_h \Vert_0
 \le 
 C\vvvert \b{e} \vvvert \Vert \pi_h p-p_h \Vert_0.
\end{align*}
Combining these estimates, we arrive at
\begin{equation}\label{eq:pressure-2}
|\mathcal{C}_h(\b{e},\b{v}^*_h)+\mathcal{R}(\b{e},\b{v}^*_h)|\le C  (\vvvert \b{e}\vvvert+\vvvert \b{u}_h-  \b{w}_h\vvvert
+h^{1/2} \Vert \b{z}- \b{\psi}_h \Vert_*)\; \Vert \pi_h p-p_h \Vert_0.
\end{equation}

 Finally, applying the inverse inequality,  discrete trace inequality \eqref{eq:discrete-trace-inequality},  and stability estimates \eqref{eq:vh*}  and  \eqref{eq:function}, we obtain
\begin{align*}
\mathcal{S}(\b{v}^*_h,\b{v}^*_h)\le & C(h^3\Vert \mr{curl}( \mathcal{L}\b{v}^*_h )\Vert^2_h+h^2\Vert (\b{b}\cdot \nabla ) \b{v}^*_h \Vert^2_{h,\mathcal{F}_h})\\
\le &
C(\nu^2 h^{-1}+h) \Vert \nabla \b{v}^*_h \Vert^2_h+Ch \Vert \b{v}^*_h \Vert^2_0 \\
\le & C (\nu^2 h^{-1}+h)\Vert \pi_h p-p_h \Vert_0^2.
\end{align*}
Thus,  we have
\begin{equation}\label{eq:pressure-3}
\mathcal{S}(\b{e},\b{v}^*_h)\le \mathcal{S}(\b{e},\b{e})^{1/2} \mathcal{S}(\b{v}^*_h,\b{v}^*_h)^{1/2}
\le  C\vvvert \b{e} \vvvert\; (\nu  h^{-1/2}+h^{1/2})\Vert \pi_h p-p_h \Vert_0.
\end{equation}
Substituting \eqref{eq:pressure-1}---\eqref{eq:pressure-3}  into \eqref{eq:pressure-0} completes the proof.
\end{proof}

\begin{remark} 
Theorem \ref{the:p-error} establishes that the error $\Vert \pi_h p - p_h \Vert_0$ is of the same order as the velocity error $\vvvert \b{u} - \b{u}_h \vvvert$ when  $\nu\le Ch^{1/2}$. 

By applying the triangle inequality, we obtain  
\begin{equation}\label{eq:p-ph}  
\Vert p - p_h \Vert_0 \leq \Vert p - \pi_h p \Vert_0 + \Vert \pi_h p - p_h \Vert_0.  
\end{equation}  
Combining this with Theorem \ref{the:p-error} and the standard approximation properties of $\pi_h$ (see, e.g., \cite{MR851383}), we derive an optimal 
 convergence rate for the pressure approximation, provided the exact pressure $p$ is sufficiently regular.  
\end{remark}

\begin{remark}
When employing the Stenberg space for velocity discretization paired with a pressure space of degree $k-1$, the theoretically achievable convergence rate for pressure approximation in the $L^2$-norm is $\mathcal{O}(h^k)$.   A notable enhancement occurs in the discrete error projection onto $Q_h$, yielding an $\mathcal{O}(h^{k+1/2})$ convergence rate. This represents a supercloseness property that exceeds standard approximation expectations for the given pressure polynomial space. 
\end{remark}

 \section{Numerical tests}
Here, we illustrate the theoretical results using the analytical solution from \cite[Example 4]{MR4331937},  defined as follows:
\begin{equation}\label{eq:Oseen-1}
\begin{aligned}
-\nu \Delta \b{u}+(\b{b}\cdot \nabla )\b{u}+c \b{u}+\nabla p =&\b{f}\quad \text{in $\Omega:=[0,1]^2$},\\
\nabla \cdot \b{u}=&0\quad \text{in $\Omega$},
\end{aligned}
\end{equation}
where $\b{b}=(b_1,b_2)^T=\b{u}+(0,1)^T$, $c=1$ and
\begin{align*}
\b{u}=&(u_1,u_2)^T=(\sin (2\pi x) \sin (2 \pi y), \cos (2\pi x) \cos (2 \pi y))^T,\\
p=&\frac{1}{4}(\cos (4 \pi x)-\cos (4 \pi y) ).
\end{align*}

All calculations are carried out
on nonuniform grids. For this aim, a sequence of shape-regular unstructured grids are generated. The coarsest grid is depicted in Figure \ref{fig:mesh level 1}.
We use  the package FreeFEM++ \cite{MR3043640} to implement the formulation \eqref{eq:Hdiv-vorticity-stabilization} with
 the Stenberg  finite element of order 2 (Stenberg2) and  piecewise affine, discontinuous  pressures. We select the jump penalization parameter as $\sigma = 6(k+1)(k+d)/d$ throughout our computations, adopting the asymptotic scaling with respect to the polynomial order $k\ge 1$ as recommended in \cite{Hillewaert:2013-Development}.
 Based on a parameter study, all simulations for convergences studies were
 performed with  $\delta_0= 10^{-5}$.  For other choices of the stabilization parameter $\delta_0$ (see log-log chart   \ref{fig:optimal stabilization coefficient}), the situation does not improve much, although
the optimum on coarse meshes seems to be slightly shifted toward larger
values. The linear systems were solved using
UMFPACK.

\begin{figure}[htbp]
\begin{minipage}[b]{0.48\linewidth}
\centering
\includegraphics[width=2.5in]{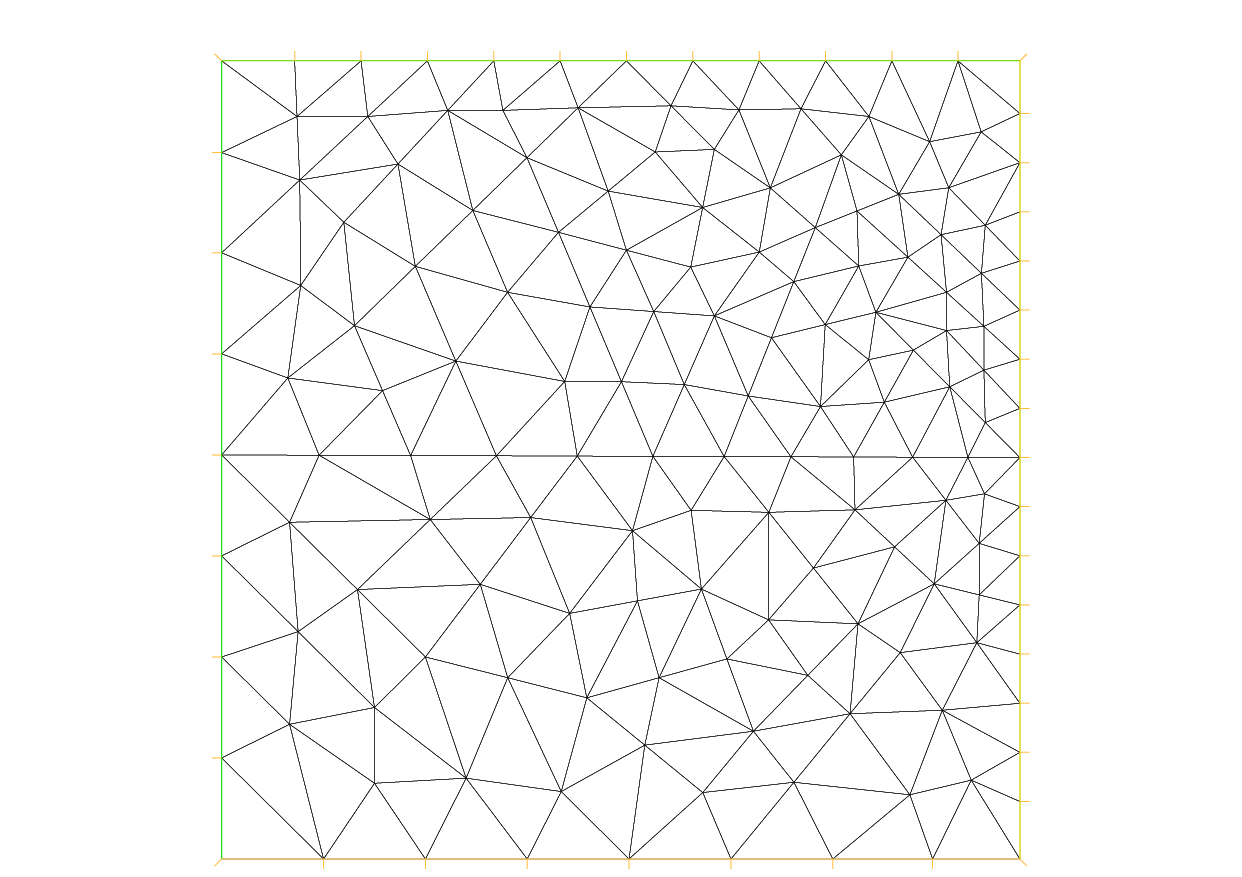}
\caption{Initial mesh level 1}
\label{fig:mesh level 1}
\end{minipage}%
\end{figure}

\begin{figure}
\centering
\includegraphics[width=4.5in]{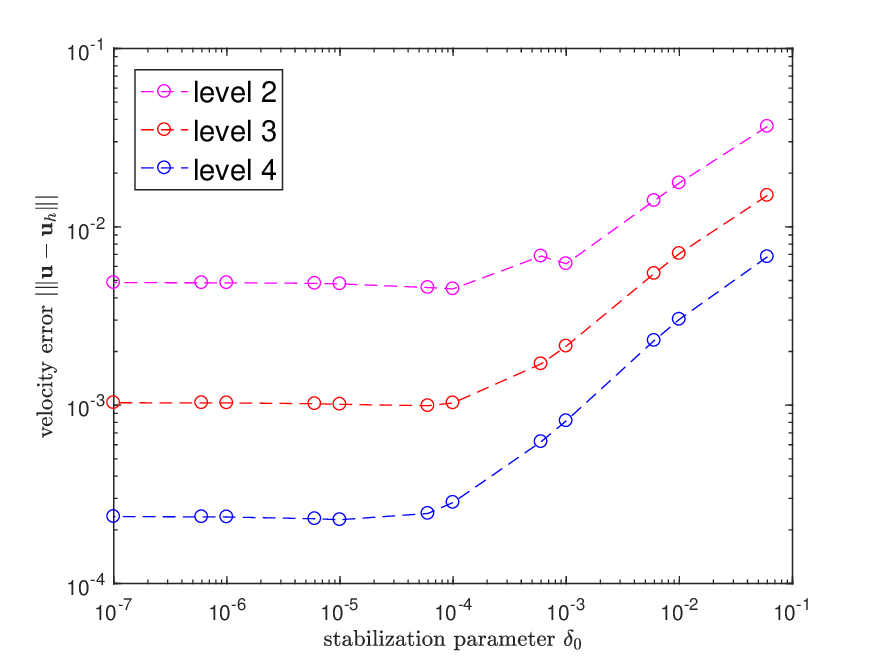}
\caption{ $\vvvert \b{u}-\b{u}_h \vvvert$ and stabilization coefficient $\delta_0$ when $\nu=10^{-6}$ }
\label{fig:optimal stabilization coefficient}
\end{figure}

\begin{figure}[htbp]
\begin{minipage}[b]{0.5\linewidth}
\centering
\includegraphics[width=2.5in]{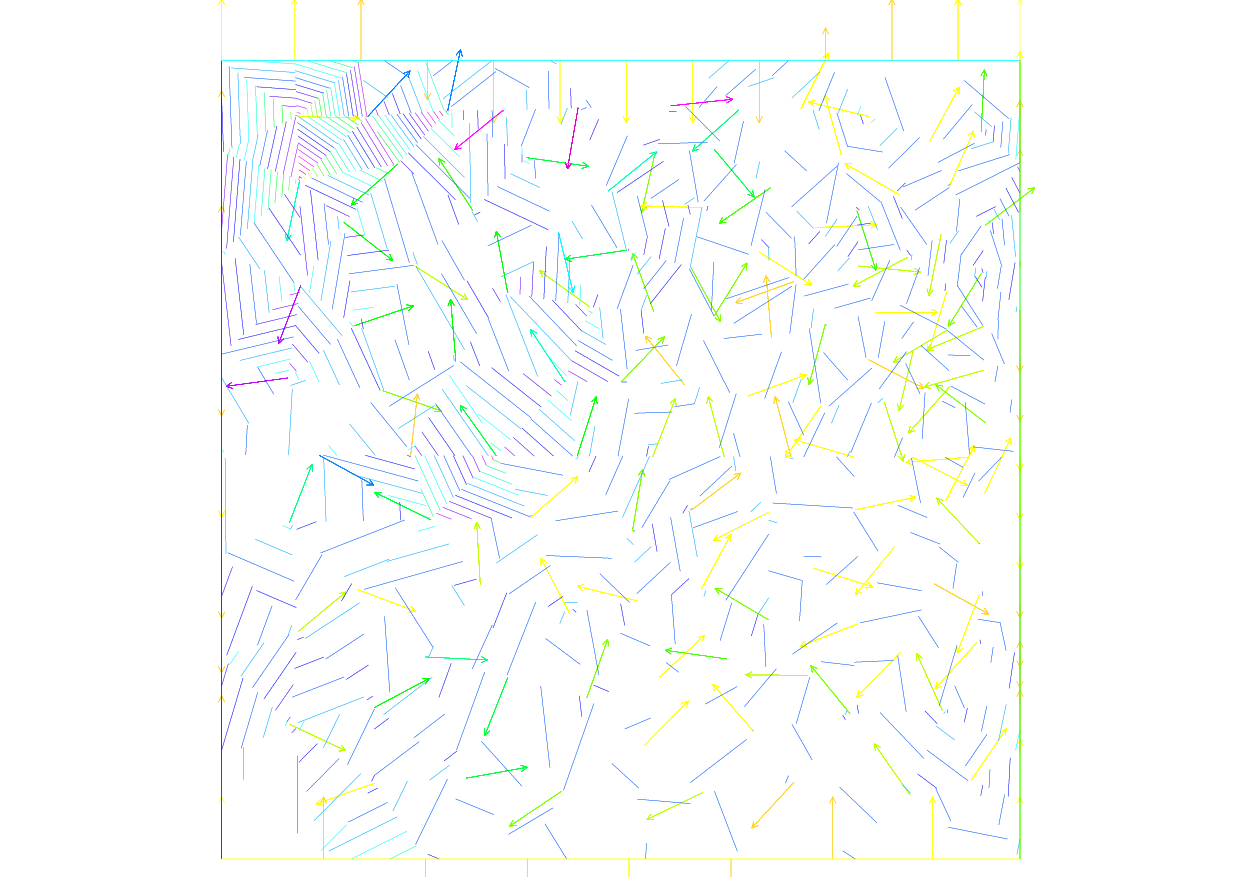}
 \captionof{figure}{vorticity stabilization.}
\label{fig:vs}
\end{minipage}
\hfill
\begin{minipage}[b]{0.5\linewidth}
\centering
\includegraphics[width=2.5in]{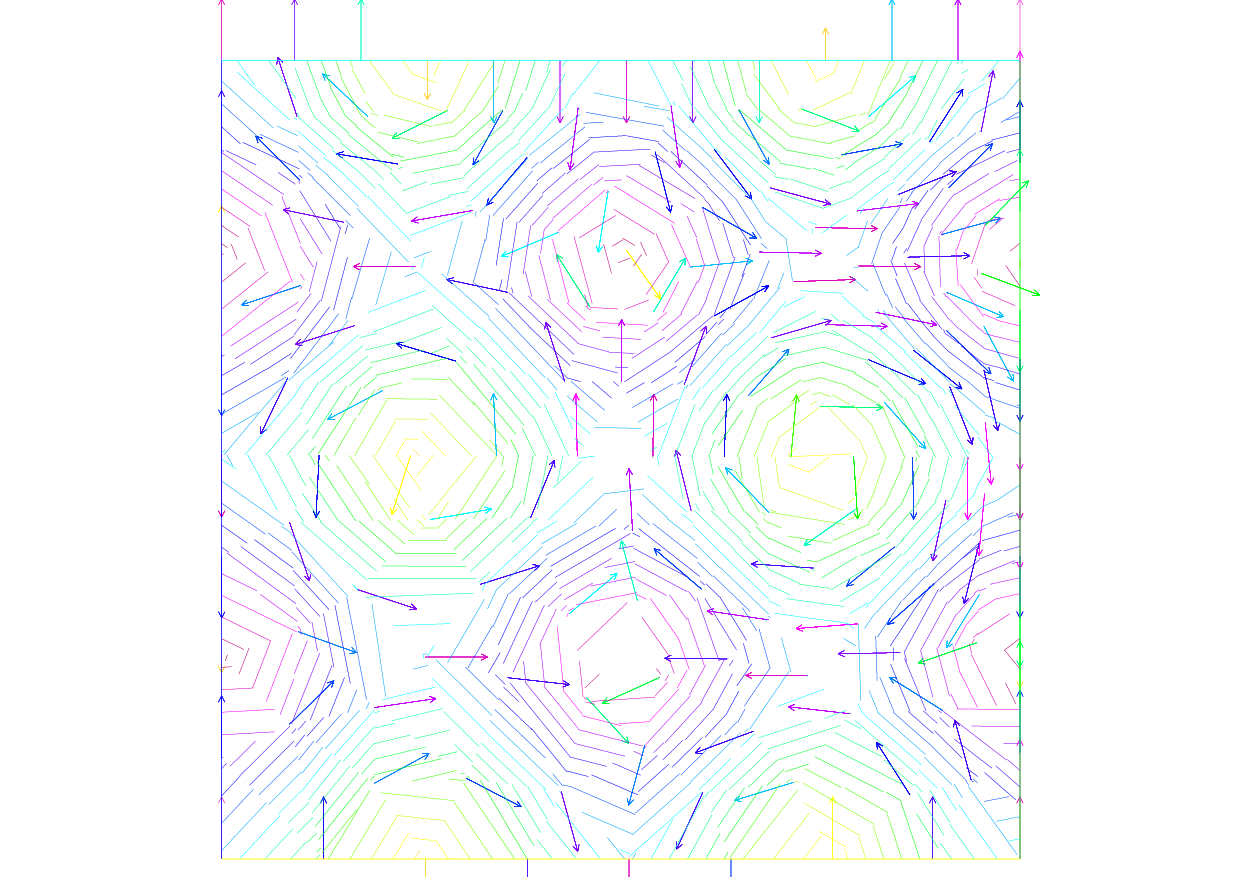}
 \captionof{figure}{upwind and vorticity stabilization.}
\label{fig:upw+vs}
\end{minipage}

%
\hfill
\begin{minipage}[b]{0.8\linewidth}
\centering
\includegraphics[width=3.5in]{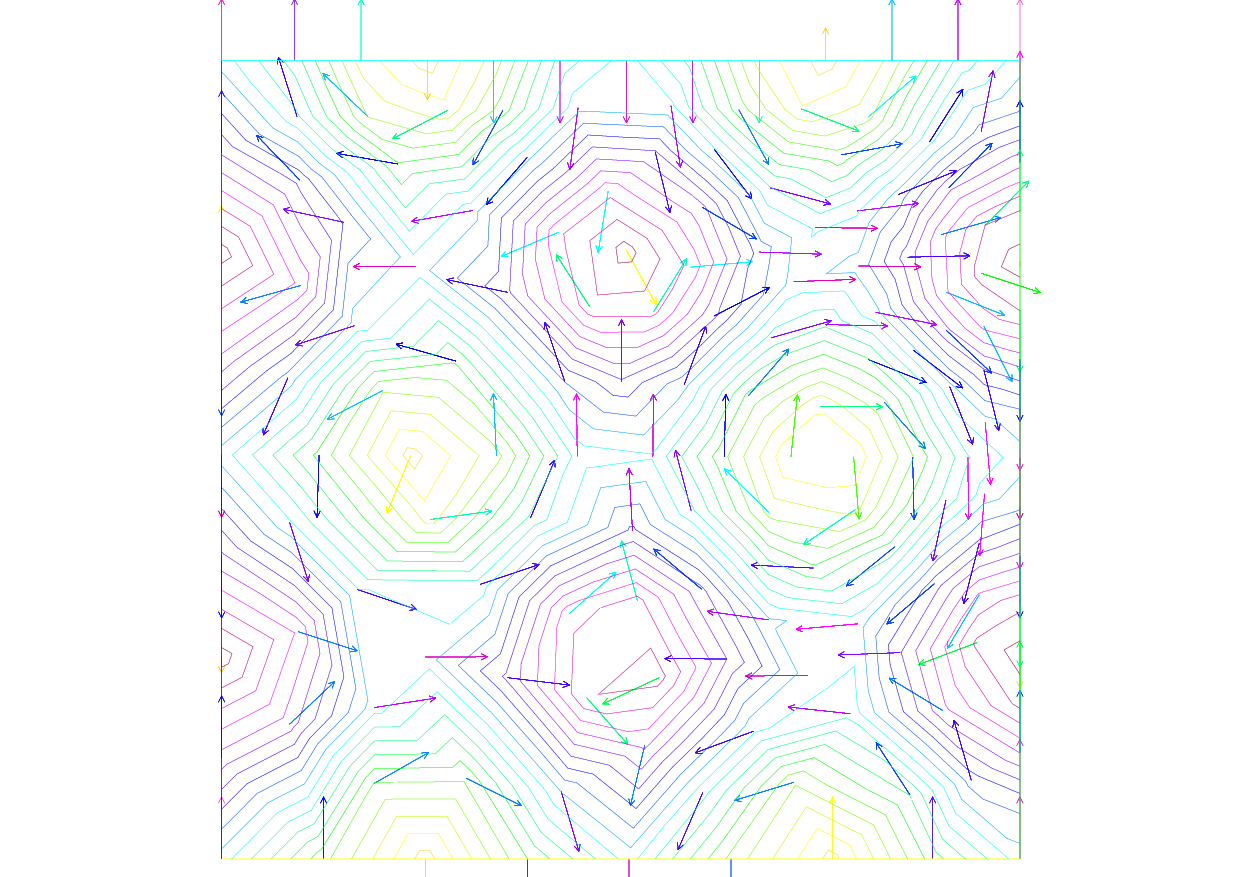}
 \captionof{figure}{True solution}
\label{fig:ts}
\end{minipage}
\caption*{Numerical solutions of formulation \eqref{eq:Hdiv-vorticity-stabilization} with different stabilization strategies on refinement level 1 when $\nu=10^{-6}$}
\end{figure}

Figures \ref{fig:vs}---\ref{fig:upw+vs} present the numerical solutions of formulation \eqref{eq:Hdiv-vorticity-stabilization} implemented with various stabilization strategies at refinement level 1. Figure \ref{fig:vs} presents the numerical solutions $\b{u}_h$ and $p_h$ obtained using only vorticity stabilization, where the term $\mathcal{C}_h(\mathbf{u}_h,\mathbf{v}_h)$ follows the definition in \eqref{eq:convection-central-flux}.  
Figure \ref{fig:upw+vs} presents the results obtained through the combined upwind and vorticity stabilization approach.

A detailed comparison with the reference solution (Figure \ref{fig:ts}) yields two key observations. First, the vorticity-only stabilization approach exhibits clear numerical instabilities, a phenomenon previously reported in the literature for nonconforming finite element methods applied to convection-dominated problems. As documented in \cite{MR1485996}, such instabilities typically require additional stabilization  strategies beyond standard SUPG-type formulations, particularly through the introduction of jump penalty terms. Second,  combined upwind-vorticity stabilization schemes maintains numerical stability throughout the simulations, showing no signs of solution oscillations or divergence.

Tables \ref{table:alex-1}--\ref{table:alex-3} present a systematic investigation of numerical errors and convergence rates for velocity and pressure fields under various norms. Tables \ref{table:alex-1} and \ref{table:alex-3} document the results obtained using Stenberg2 elements with combined upwind and vorticity stabilization for $\nu=10^{-6}$ and $\nu=1$ cases. 
These comprehensive numerical experiments provide strong validation of the theoretical result established in Corollary \ref{the:u-uh}.  

\begin{table}[htbp]
\caption{Errors and convergence rates for the stabilized formulation   \eqref{eq:Hdiv-vorticity-stabilization} using  combined upwind and vorticity stabilization when $\nu=10^{-6}$   }
\footnotesize
\begin{tabular*}{\textwidth}{@{\extracolsep{\fill}} c |ccccc }
\cline{1-6}{}
\diagbox{}{ Mesh }  & level 1 &level 2 & level 3 &level 4 &level 5\\
\cline{1-6}
$\vvvert  \b{u}-\b{u}_h\vvvert$    &1.86E-2  &4.79E-3  &1.01E-3 &2.28E-4  &4.73E-5 \\
 order    &1.95  &2.48  &2.24 &2.96  &--- \\ 
 $\Vert  \b{u}-\b{u}_h\Vert_0$  & 8.95E-3  &2.14E-3 &3.54E-4 &6.72E-5  &1.01E-5 \\
order    &2.05  &2.87  &2.50&3.57 &---\\
$\Vert \nabla \cdot  \b{u}_h \Vert_0$ &8.62E-4 &4.02E-5  &2.68E-6 &7.39E-7 &2.13E-8 \\
order   &4.39 &4.33  &1.93 &6.69&---\\
$\Vert \b{u}-\b{u}_h \Vert_{0,\infty}$  & 3.42E-2&1.82E-2   &2.59E-3&9.34E-4 &1.36E-4\\
order  & 0.90 & 3.11  & 1.53&3.63 &---\\
$\Vert p-p_h \Vert_0$ & 1.18E-2& 2.99E-3  &7.03E-4 &1.69E-4 &4.11E-5\\
order &1.96 &2.31   &2.14 &2.67&---\\
\cline{1-6}
\end{tabular*}
\label{table:alex-1}
\end{table}

%

\begin{table}[htbp]
\caption{Errors and convergence rates for the stabilized formulation   \eqref{eq:Hdiv-vorticity-stabilization} using  combined upwind and vorticity stabilization when $\nu=1$   }
\footnotesize
\begin{tabular*}{\textwidth}{@{\extracolsep{\fill}} c |ccccc }
\cline{1-6}{}
\diagbox{}{ Mesh }   & level 1 &level 2 & level 3 &level 4 &level 5\\
\cline{1-6}
$\vvvert  \b{u}-\b{u}_h\vvvert$      &5.54E+0&4.96E-1  &1.90E-2  &4.93E-3 &1.23E-3   \\
 order      &3.46  &5.21 &2.02 &2.61  &--- \\  
 $\Vert  \b{u}-\b{u}_h\Vert_0$    &3.40E-1& 2.40E-2  &9.76E-5 &8.16E-6 &9.94E-7   \\
order      &3.79  &8.79  &3.73&3.97 &---\\
$\Vert \nabla \cdot  \b{u}_h \Vert_0$     &3.31E-3 &1.23E-4  &2.03E-5 &1.68E-6 &1.55E-7 \\
order     &4.72 &2.88  &3.74 &4.49&---\\
$\Vert \b{u}-\b{u}_h \Vert_{0,\infty}$    & 8.45E-1&4.74E-2   &5.74E-4&3.75E-5 &5.66E-6\\
order    & 4.13 & 7.05  & 4.10&3.56 &---\\
$\Vert p-p_h \Vert_0$     & 1.86E+1& 1.95E+1  &6.79E-2 &1.82E-2 &4.56E-3\\
order  &-0.7 &9.04   &1.98 &2.60&---\\
\cline{1-6}
\end{tabular*}
\label{table:alex-3}
\end{table}


To compare the computational efficiency between Stenberg2 and  the BDM  finite element of order 2 (BDM2) elements with upwind stabilization, we tabulate both the DOFs and computation time in Tables \ref{tab:Stenberg-ndof} and \ref{tab:BDM2-ndof}. The data reveals that Stenberg2 maintains approximately 20\% fewer DOFs than BDM2 across all refinement levels. This reduced dimensionality translates to progressively greater computational time savings for Stenberg2 as the problem size increases, with the efficiency advantage becoming more pronounced at higher DOF counts.

\begin{table}[htbp]
\begin{minipage}[b]{0.48\linewidth}
  \centering
  \caption{Stenberg2}
  \label{tab:Stenberg-ndof}
  \begin{tabular}{c|ccc|c}
    \toprule
    \textbf{level} & \textbf{ndof $\mathbf{u}_h$} & \textbf{ndof $p_h$} & \textbf{total ndof} & \textbf{CPU time}\\
    \midrule
    1 & 1575 & 822   & 2397   &0.23\\
    2 & 6250 & 3336   & 9586   &0.97\\
    3 & 24752 & 13356 & 38108  &4.257\\
    4 & 98254 & 53304 & 151558 &20.302\\
    5 & 394264 & 214476 & 608740 &108.565\\
    \bottomrule
  \end{tabular}
\end{minipage}

\begin{minipage}[b]{0.48\linewidth}
  \centering
  \caption{BDM2}
  \label{tab:BDM2-ndof}
  \begin{tabular}{c|ccc|c}
    \toprule
    \textbf{level} & \textbf{ndof $\mathbf{u}_h$} & \textbf{ndof $p_h$} & \textbf{total ndof} & \textbf{CPU time}\\
    \midrule
    1 & 2121 & 822   & 2943   &0.2571\\
    2 & 8472& 3336   & 11808   &1.118\\
    3 & 33654 & 13356 & 47010  &5.073\\
    4 & 133788 & 53304 &187092 &25.561\\
    5 & 537246 & 214476 & 751722 &163.825\\
    \bottomrule
  \end{tabular}
\end{minipage}

%
\end{table}

\appendix

\section{Proof of Lemma \ref{lem:stability-stenberg} }
To prove 
\begin{equation}\label{eq:Stenberg-stability}
\inf\limits_{q_h\in Q_h}\sup\limits_{\b{v}_h\in V_h}\frac{b( \b{v}_h,q_h)}{\Vert \b{v}_h \Vert_{1,h} \Vert q_h \Vert_{0}}=\beta_{is}^h>0,
\end{equation}
 we will 
  employ the macro-element analysis technique introduced in \cite{MR725982}. First, we introduce some notations. 
  Let $M=K_1\cup K_2$ be an arbitrary macroelement consisting of two adjacent elements $K_1$ and $K_2$ sharing a common edge (in 2D) or face (in 3D). The space  $Q_h$ admits the orthogonal decomposition: 
$$
Q_h=\bigoplus_M Q_h(M) \oplus \bar{Q}_h
$$
where
\begin{align*}
&Q_h(M)=\{q_h\in  L^2_0(\Omega):\;q_h|_{K_i}\in P_{k-1}(K_i),\; i=1,2;\;q_h=0\;\text{on $\Omega-M$} \},\\
&\bar{Q}_h=\left\{ q_h\in L^2_0(\Omega):\;q_h|_M\in P_0(M)\right\}.
\end{align*}
Define the local velocity space on $M$ as: 
$$
V_h(M)=\left\{
\begin{aligned}
& \b{v}\in H_0(\mr{div},M):\;\b{v}|_{K_i}\in \left( P_k(K_i) \right)^d,\;i=1,2;\\
&\text{all degrees of freedom on  $\partial M$ vanish}; \;\b{v}=\b{0}\;\text{on $\Omega-M$}
\end{aligned}
\right\}.
$$
Then, the  stability \eqref{eq:Stenberg-stability} can be established through verification of two key conditions: first,  the uniform inf-sup stability of the pair \( V_h(M)/Q_h(M) \) with respect to $M$, and second, the inf-sup stability of the pair \( V_h/\bar{Q}_h \).


\subsection{(a) Uniform inf-sup stability for \( V_h(M) / Q_h(M) \)}
Given the quasi-uniformity of $\mathcal{T}_h$, the finite element pair \( V_h(M)/Q_h(M) \) is uniformly inf-sup stable with respect to 
$M$  if and only if the following macroelement condition (cf.  \cite{MR725982}, \cite
[\S 3.5.4]{MR3561143}) holds :
\begin{equation}\label{eq:MC}
N_M := \left\{ q \in Q_h(M) : (\mathrm{div}\, \mathbf{v}, q)_M = 0 \quad \forall \mathbf{v} \in V_h(M) \right\} = \{0\}.
\end{equation}
Let \( q|_{K_i} = \sum_{|\beta|=0}^{k-1} c_\beta^i \mathbf{x}^\beta \). 
For each \( i = 1, 2 \) and multi-index \( \alpha \) with \( |\alpha| \geq 1 \), we construct \( \mathbf{v}_i^\alpha \in V_h(M) \) whose support is contained in \( K_i \) and which satisfies
\[
\int_{K_i} \mathbf{v}_i^\alpha \cdot \nabla \mathbf{x}^\alpha \mr{d}\b{x} = 1,
\]
as its sole non-zero degree of freedom; cf. \cite[(2.9), (2.10), (2.12)--(2.16)]{MR2594344}.
Then, integration by parts yields:  
\[
0 = (\mathrm{div}\, \mathbf{v}_i^\alpha, q)_M = -\sum_{|\beta|=0}^{k-1} c_\beta^i \int_{K_i} \mathbf{v}_i^\alpha \cdot \nabla \mathbf{x}^\beta \mr{d}\b{x} = -c_\alpha^i.
\]  
Thus, \( c_\alpha^i = 0 \) for \( |\alpha| \ge 1 \), implying \( q|_{K_i} = c_i \) (constant on each \( K_i \)).  Further, for any \( \mathbf{v} \in V_h(M) \),  
\[
0 = (\mathrm{div}\, \mathbf{v}, q)_M = (c_1 - c_2) \int_F \mathbf{v} \cdot \mathbf{n}_F \, \mathrm{d}s, \quad F = K_1 \cap K_2.
\]  
Select \( \mathbf{v} \in V_h(M) \) such that \( \int_F \mathbf{v} \cdot \mathbf{n}_F \, \mathrm{d}s = 1 \). This forces \( c_1 = c_2 \), so \( q|_M \) is globally constant. Since \( q \in L^2_0(M) \), we conclude \( q = 0 \).  Thus the macroelement condition \eqref{eq:MC} holds.

\subsection{(b) Inf-sup stability for \( V_h / \bar{Q}_h \) } 
The inf-sup stability of the pair \( V_h / \bar{Q}_h \) holds if and only if there exists a Fortin operator \( \pi : V \to V_h \) satisfying the following conditions:
\begin{equation}\label{eq:MT-2}
\begin{aligned}
&(\nabla \cdot (\pi \mathbf{v}), \bar{q}_h) = (\nabla \cdot \mathbf{v}, \bar{q}_h), && \forall \bar{q}_h \in \bar{Q}_h, \\
&\|\pi \mathbf{v}\|_{1,h} \leq C \|\mathbf{v}\|_1, && \forall \mathbf{v} \in V.
\end{aligned}
\end{equation}
To ensure the first condition in \eqref{eq:MT-2}, we impose the flux continuity requirement:
\begin{equation}\label{eq:ME-Fortin-C1}
\int_{T} (\b{v}-\pi \b{v})\cdot \b{n} \mr{d}s=0,
\end{equation}
     where \( T= M \cap M' \) denotes the interface between adjacent  macroelements \( M \) and \( M' \).  
For the  stability of the Fortin operator, we  prescribe the vertex-based averaging:
\begin{equation}\label{eq:ME-Fortin-C2}
     (\pi \mathbf{v})(\mathbf{x}_i) = \frac{1}{|S_i|} \int_{S_i} \mathbf{v},  
\end{equation}
     where \( \mathbf{x}_i \) are vertices of \( K \in \mathcal{T}_h \).  
All remaining degrees of freedom for $\pi\b{v}$ follow the conventional definition of   Stenberg elements \cite{MR2594344}, with the modification that the degree of freedom corresponding to (2.15) in \cite{MR2594344}  must be expressed in integral form. This completes the full specification of the operator \( \pi \mathbf{v} \), which satisfies the first equation in \eqref{eq:MT-2}. 
     
     The stability estimate  \(\Vert \pi v \Vert_{1,h} \le C \vert v \vert_1\) follows from the definition of the projection operator $\pi$ combined with the Poincar\'e inequality.

   We focus on the two-dimensional case. Let $\hat{K}$ be the reference element with vertices $\hat{\b{V}}_0, \hat{\b{V}}_1, \hat{\b{V}}_2$ and edges $\hat{e}_0, \hat{e}_1, \hat{e}_2$, where $\hat{\b{n}}^i$ denotes the unit outer normal vector to the edge $\hat{e}_i$. Consider a mesh element $K = \mathscr{F}(\hat{K})$ obtained through an affine mapping $\mathscr{F}$, with corresponding vertices $\b{V}_i = \mathscr{F}(\hat{\b{V}}_i)$ and edges $e_0, e_1, e_2$. The unit outer normal vector to $e_i \subset \partial K$ is denoted by $\b{n}^i$.

The dual basis $\hat{\b{\omega}}(\hat{\b{x}})$ is defined with respect to the degrees of freedom $\hat{\mathcal{D}}$ on the reference element  $\hat{K}$ as follows:
\begin{equation}\label{eq:DOF-hatK-Stenberg-2D}
\begin{aligned}
&\hat{\b{f}}(\hat{\b{V}}_i), \quad i=0,1,2, \\
&\int_{\hat{e}_i} \hat{\b{f}} \cdot \hat{\b{n}}^i \, \hat{\phi}_j^i \, \mr{d}\hat{s}, \quad i=0,1,2, \quad \hat{\phi}_j^i \in \mathbb{P}_{k-2}(\hat{e}_i), \\
&\int_{\hat{K}} \hat{\b{f}} \cdot \hat{\b{\psi}}_i \, \mr{d}\hat{\b{x}}, \quad \hat{\b{\psi}}_i \in \mathcal{N}_{k-2}(\hat{K}),
\end{aligned}
\end{equation}
where $\{\hat{\phi}_j^i(\hat{\b{x}})\}$ forms a basis for the polynomial space $\mathbb{P}_{k-2}(\hat{e}_i)$, and $\{\hat{\b{\psi}}_i(\hat{\b{x}})\}$ constitutes a basis for the  first
kind Nédélec space $\mathcal{N}_{k-2}(\hat{K})$, whose definition is referred to \cite[(2.3.37)]{Bof1Bre2For3:2013-Mixed}. Similarly, the dual basis $\b{\omega}(\b{x})$ on the physical element $K$ corresponds to the degrees of freedom $\mathcal{D}$:
\begin{equation}\label{eq:DOF-Stenberg-2D}
\begin{aligned}
&\b{f}(\b{V}_i), \quad i=0,1,2, \\
&\int_{e_i} \b{f} \cdot \b{n}^i \, \phi_j^i \, \mr{d}s, \quad i=0,1,2, \quad \phi_j^i \in \mathbb{P}_{k-2}(e_i), \\
&\int_K \b{f} \cdot \b{\psi}_i \, \mr{d}\b{x}, \quad \b{\psi}_i \in \mathcal{N}_{k-2}(K),
\end{aligned}
\end{equation}
with the basis functions transforming according to:
\begin{equation}\label{eq: bianhuan-I}
\begin{aligned}
\phi_j^i(\b{x}) &= \hat{\phi}_j^i(\hat{\b{x}}), \quad \text{where } \hat{\b{x}} = \mathscr{F}^{-1}(\b{x}), \\
\b{\psi}_i(\b{x}) &= \left(\frac{\partial \hat{\b{x}}}{\partial \b{x}}\right)^T \hat{\b{\psi}}_i(\hat{\b{x}}).
\end{aligned}
\end{equation}

 Assume
 \begin{equation}\label{eq:relation-K-hatK}
 \begin{aligned}
\b{\omega}(\b{x})=&\mr{det}\left(\frac{\partial \hat{\b{x}}}{\partial \b{x}} \right)\left(\frac{\partial \hat{\b{x}}}{\partial \b{x}} \right)^{-1}\hat{\b{s}}(\hat{\b{x}}),\\
\hat{\b{s}}(\hat{\b{x}})=&\hat{\underline{\b{\omega}}}(\hat{\b{x}})\cdot \b{y},
\end{aligned}
\end{equation}
where 
$$
\hat{\underline{\b{\omega}}}(\hat{\b{x}})=(\hat{\b{\omega}}^V_0(\hat{\b{x}}),\hat{\b{\omega}}^V_1(\hat{\b{x}}), \hat{\b{\omega}}^V_2(\hat{\b{x}}),  \hat{\b{\omega}}^e_{0,0}(\hat{\b{x}}), \ldots, \hat{\b{\omega}}^e_{2,k-2}(\hat{\b{x}}),\hat{\b{\omega}}^C_{0}(\hat{\b{x}}),\ldots)^T
$$ is the vector-valued function consisting of the dual basis functions  $\{ \hat{ \b{\omega} }_i (\hat{\b{x}})\}$  and $\b{y}$ is the coordinate vector of $\hat{\b{s}}(\hat{\b{x}})$.

From \eqref{eq: bianhuan-I} and \eqref{eq:relation-K-hatK}, we have
\begin{equation}\label{eq:Stenberg2-DOF-II}
\begin{aligned}
&\b{\omega}(\b{V}_i)=\mr{det}\left(\frac{\partial \hat{\b{x}}}{\partial \b{x}} \right)\left(\frac{\partial \hat{\b{x}}}{\partial \b{x}} \right)^{-1} \hat{\b{s}}(\hat{ \b{V}}_i )\quad \forall i,\\
&\int_{e_i}\b{\omega}\cdot \b{n}^i\; \phi_j^i \mr{d}s= \int_{\hat{e}_i}\hat{\b{s}}\cdot \hat{\b{n}}^i \; \hat{\phi}_j^i \mr{d}\hat{s} \quad \forall i,j,\\
&\int_K \b{\omega}\cdot  \b{\psi}_i\mr{d}\b{x}=
\int_{\hat{K}} \hat{ \b{s} }\cdot  \hat{ \b{\psi}}_i\mr{d}\hat{ \b{x} }
\quad \forall i.
\end{aligned}
\end{equation}     
Then 
we have
\begin{equation}\label{eq:Stenberg2-DOF-III}
\b{y}=B_K\b{b},
\end{equation}
where 
$$
B_K=\left(
\begin{matrix}
\mr{det}\left(\frac{\partial \b{x} } {\partial \hat{\b{x}}}\right)
\left(\frac{\partial \b{x}}{\partial \hat{\b{x}} } \right)^{-1}&\b{0}&\b{0}&\b{0}\\
\b{0}&\mr{det}\left(\frac{\partial \b{x} } {\partial \hat{\b{x}}}\right)
\left(\frac{\partial \b{x}}{\partial \hat{\b{x}} } \right)^{-1}&\b{0}&\b{0}\\
\b{0}&\b{0}&\mr{det}\left(\frac{\partial \b{x} } {\partial \hat{\b{x}}}\right)
\left(\frac{\partial \b{x}}{\partial \hat{\b{x}} } \right)^{-1}&\b{0}\\
\b{0}&\b{0}&\b{0}&\b{I}
\end{matrix}
\right),
$$ 
and
$$
\b{b}=\left( \b{\omega}(\b{V}_0),
 \b{\omega}(\b{V}_1),
 \b{\omega}(\b{V}_2),
\int_{e_0}\b{\omega}\cdot \b{n}^0\; \phi^0_0 \mr{d}s,
\ldots,
\int_{e_2}\b{\omega}\cdot \b{n}^2\; \phi^2_{k-2} \mr{d}s,
\int_K \b{\omega}\cdot  \b{\psi}_0\mr{d}\b{x},
\ldots
\right)^T.
$$

The operator $\pi$ is defined as follows:
\begin{equation}\label{eq:definition-pi}
\begin{aligned}
&\pi \b{v}= \sum_{i=0}^{2} \frac{1}{|S_i|}\int_{S_i}\b{v}(\b{x})\mr{d}\b{x}\;\b{\omega}^V_i(\b{x})\\
&+\sum_{i=0}^{2} \sum_{j=0}^{k-2}\int_{e_i} (\b{v}\cdot \b{n}^i) \phi_j^i \mr{d}s\; \b{\omega}^e_{i,j}(\b{x})+\sum_{i} \int_{K} \b{v}\cdot  \b{\psi}_i\mr{d}\b{x}\; \b{\omega}_{i}^C(\b{x}),
\end{aligned}
\end{equation}
where $S_i$ is the union of mesh elements containing the vertex $V_i$ and $\b{\omega}^V_i, \b{\omega}^e_{i,j},\b{\omega}_{i}^C$ represent the dual basis functions corresponding to the degrees of freedom specified in \eqref{eq:DOF-Stenberg-2D}.

Now we need to prove the following stability: for $m=0,1$
\begin{equation}\label{eq:I}
\Vert D^m \pi \b{v} \Vert_{0,K}
\le C\left( h^{-m} \Vert  \b{v} \Vert_{0,S(K)} +h^{1-m} \Vert  \nabla\b{v} \Vert_{0,S(K)} \right),
\end{equation}
where $S(K)$ is the set of the influence elements of $K$, i.e.,
$$
S(K)=\bigcup_{T\in\mathcal{I}(K)}T,\quad \mathcal{I}(K):=\{T\in\mathcal{T}_h:\;T\cap K\ne \emptyset\}.
$$
From Cauchy-Schwarz inequality, the continuous trace  inequality \eqref{eq:trace-inequality} and \eqref{eq: bianhuan-I}, we have
\begin{equation}\label{eq:I-1}
\begin{aligned}
&\left|  \frac{1}{|S_i|}\int_{S_i}\b{v}(\b{x})\mr{d}\b{x}\right|\le C h^{-1} \Vert \b{v} \Vert_{0,S(K)},\\
& \left| \int_{e_i} (\b{v}\cdot \underline{\b{n}}^i) \phi_j \mr{d}s\right|
 \le  \Vert   \b{v}  \Vert_{0,e_i} \Vert  \phi_j  \Vert_{0,e_i}\le C \left(   \Vert   \b{v}  \Vert_{0,K}
+h \Vert   \nabla \b{v}  \Vert_{0,K} \right),\\
&\left| \int_{K} \b{v}\cdot  \b{\psi}_i\mr{d}\b{x}\right| \le
\Vert \b{v} \Vert_{0,K} \Vert \b{\psi}_i \Vert_{0,K}
\le 
C\Vert   \b{v} \Vert_{0,K}.
\end{aligned}
\end{equation}
From \eqref{eq:Stenberg2-DOF-III} and  \eqref{eq:relation-K-hatK}, one has
\begin{equation}\label{eq:I-2}
\begin{aligned}
&\Vert D^m\b{\omega}^V_i(\b{x})\Vert_{0,K}\le C h^{1-m},\\
&\Vert D^m\b{\omega}^e_{i,j}(\b{x})\Vert_{0,K}+\Vert D^m\b{\omega}^C_i(\b{x})\Vert_{0,K}\le C h^{-m}.
\end{aligned}
\end{equation}
From \eqref{eq:definition-pi}, \eqref{eq:I-1} and  \eqref{eq:I-2}, we can easily obtain \eqref{eq:I}.

Define
$$
\b{c}_K=\frac{1}{|S(K)|}\int_{S(K)}\b{v}(\b{x})\mr{d}\b{x},
$$
and note 
$$
\int_{S(K)} (\b{v}-\b{c}_K)\mr{d}\b{x}=0.
$$
From \eqref{eq:I} and  the Poincar{\'e} inequality \cite[Lemma C.3]{MR3561143}, we have
\begin{equation}\label{eq:II}
\begin{aligned}
&\Vert \nabla \pi \b{v} \Vert_{0,K}
=\Vert \nabla (\pi (\b{v}-\b{c}_K) ) \Vert_{0,K}\\
 \le &C\left( h^{-1} \Vert  \b{v}-\b{c}_K \Vert_{0,S(K)} +  \Vert  \nabla (\b{v}-\b{c}_K) \Vert_{0,S(K)} \right)\\
  \le &C\Vert  \nabla \b{v} \Vert_{0,S(K)}.
\end{aligned}
\end{equation}
We now analyze the term $\sum_{F\in\mathcal{F}_h} h^{-1}_F \Vert [\pi \b{v}  ] \Vert^2_{0,F}$. The definition of the auxiliary set $S(F)$ depends on the face type: for interior faces $\mathcal{F}_h^i\ni F = K \cap K'$ between adjacent elements $K$ and $K'$, we set $S(F) = S(K) \cup S(K')$, while for boundary faces $F \in \mathcal{F}_h^\partial$, we define $S(F) = S(K)$. The mean value of $\b{v}$ over $S(F)$, denoted by $\b{c}_F$, is given by
$$
\b{c}_F=\frac{1}{|S(F)|}\int_{S(F)}\b{v}(\b{x})\mr{d}\b{x}.
$$
This definition immediately yields the zero-mean property:
$$
\int_{S(F)} (\b{v}-\b{c}_F)\mr{d}\b{x}=0.
$$
For any $F\in\mathcal{F}^i_h$, discrete trace inequality \cite[Lemma 1.46]{MR2882148}, \eqref{eq:I}  and the Poincar{\'e} inequality \cite[Lemma C.3]{MR3561143} yield 
\begin{equation}\label{eq:III}
\begin{aligned}
& h^{-1}_F \Vert [\pi \b{v}  ] \Vert^2_{0,F}
=h^{-1}_F \Vert [\pi \b{v}-\b{c}_F  ] \Vert^2_{0,F}
=h^{-1}_F \Vert [\pi (\b{v}-\b{c}_F)  ] \Vert^2_{0,F} \\
\le & Ch^{-2}  \left( \Vert  \pi (\b{v}-\b{c}_F)     \Vert^2_{0,K}+\Vert  \pi (\b{v}-\b{c}_F)    \Vert^2_{0,K'} \right)\\
\le &
C\left( h^{-2} \Vert  \b{v}-\b{c}_F \Vert_{0,S(K)}^2 +  \Vert  \nabla (\b{v}-\b{c}_F) \Vert_{0,S(K)}^2 \right)\\
&+C\left( h^{-2} \Vert  \b{v}-\b{c}_F \Vert_{0,S(K')}^2 +  \Vert  \nabla (\b{v}-\b{c}_F) \Vert_{0,S(K')}^2 \right)\\
\le &
C\left( h^{-2} \Vert  \b{v}-\b{c}_F \Vert_{0,S(F)}^2 +  \Vert  \nabla \b{v} \Vert_{0,S(F)}^2 \right)\\
\le &
C \Vert  \nabla \b{v} \Vert_{0,S(F)}^2.
\end{aligned}
\end{equation}
For any $F\in\mathcal{F}^{\partial}_h$, , the following estimate holds:
$$
|c_F|\le  C \Vert \nabla \b{v} \Vert_{0,S(F)}.
$$
This result follows immediately from the vanishing trace condition  $\b{v}|_F=\b{0}$ and  the Poincar{\'e} inequality \cite[Theorem A.36 \& Remark A.37]{MR3561143}. Furthermore, by combining this estimate with techniques similar to those used in establishing inequality \eqref{eq:III}, we can derive the following bound
\begin{equation}\label{eq:IV}
\begin{aligned}
h^{-1}_F \Vert [\pi \b{v}  ] \Vert^2_{0,F}
\le C \left(h^{-1}_F \Vert \pi (\b{v}-\b{c}_F)  \Vert^2_{0,F}+ |\b{c}_F|^2\right) 
\le
C \Vert  \nabla \b{v} \Vert_{0,S(F)}^2.
\end{aligned}
\end{equation}

Gathering the estimates from \eqref{eq:II}, \eqref{eq:III} and \eqref{eq:IV} and summing over all elements \( K \in \mathcal{T}_h \) and all facets $F\in\mathcal{F}_h$, we establish the stability inequality  in \eqref{eq:MT-2}.



\bibliographystyle{plain}

\begin{thebibliography}{10}

\bibitem{MR4331937}
N.~Ahmed, G.~R. Barrenechea, Erik Burman, J.~Guzm\'{a}n, A.~Linke, and
  C.~Merdon.
\newblock A pressure-robust discretization of {O}seen's equation using
  stabilization in the vorticity equation.
\newblock {\em SIAM J. Numer. Anal.}, 59(5):2746--2774, 2021.

\bibitem{MR3908678}
D.~N. Arnold.
\newblock {\em Finite Element Exterior Calculus}, volume~93 of {\em CBMS-NSF
  Regional Conference Series in Applied Mathematics}.
\newblock Society for Industrial and Applied Mathematics (SIAM), Philadelphia,
  PA, 2018.

\bibitem{MR2269741}
D.~N. Arnold, R.~S. Falk, and R.~Winther.
\newblock Finite element exterior calculus, homological techniques, and
  applications.
\newblock {\em Acta Numer.}, 15:1--155, 2006.

\bibitem{MR2594630}
D.~N. Arnold, R.~S. Falk, and R.~Winther.
\newblock Finite element exterior calculus: from {H}odge theory to numerical
  stability.
\newblock {\em Bull. Amer. Math. Soc. (N.S.)}, 47(2):281--354, 2010.

\bibitem{MR4098216}
G.~Barrenechea, E.~Burman, and J.~Guzm\'{a}n.
\newblock Well-posedness and {$H(\rm div)$}-conforming finite element
  approximation of a linearised model for inviscid incompressible flow.
\newblock {\em Math. Models Methods Appl. Sci.}, 30(5):847--865, 2020.

\bibitem{Bof1Bre2For3:2013-Mixed}
D.~Boffi, F.~Brezzi, and M.~Fortin.
\newblock {\em Mixed Finite Element Methods and Applications}, volume~44 of
  {\em Springer Series in Computational Mathematics}.
\newblock Springer, Heidelberg, 2013.

\bibitem{MR1974504}
S.~C. Brenner.
\newblock Poincar\'{e}-{F}riedrichs inequalities for piecewise {$H^1$}
  functions.
\newblock {\em SIAM J. Numer. Anal.}, 41(1):306--324, 2003.

\bibitem{MR890035}
F.~Brezzi, J.~Douglas, Jr., R.~Dur\'{a}n, and M.~Fortin.
\newblock Mixed finite elements for second order elliptic problems in three
  variables.
\newblock {\em Numer. Math.}, 51(2):237--250, 1987.

\bibitem{MR859922}
F.~Brezzi, J.~Douglas, Jr., and L.~D. Marini.
\newblock Recent results on mixed finite element methods for second order
  elliptic problems.
\newblock In {\em Vistas in applied mathematics}, Transl. Ser. Math. Engrg.,
  pages 25--43. Optimization Software, New York, 1986.

\bibitem{chen2023finite}
L.~Chen and X.~H. Huang.
\newblock Finite element complexes in two dimensions, 2023.

\bibitem{MR4654617}
L.~Chen and X.~H. Huang.
\newblock Finite element de {R}ham and {S}tokes complexes in three dimensions.
\newblock {\em Math. Comp.}, 93(345):55--110, 2024.

\bibitem{MR3802677}
S.~H. Christiansen, J.~Hu, and K.~B. Hu.
\newblock Nodal finite element de {R}ham complexes.
\newblock {\em Numer. Math.}, 139(2):411--446, 2018.

\bibitem{MR2609313}
M.~Costabel and A.~McIntosh.
\newblock On {B}ogovski\u{\i} and regularized {P}oincar\'{e} integral operators
  for de {R}ham complexes on {L}ipschitz domains.
\newblock {\em Math. Z.}, 265(2):297--320, 2010.

\bibitem{MR2882148}
D.~A. Di~Pietro and A.~Ern.
\newblock {\em Mathematical Aspects of Discontinuous {G}alerkin Methods},
  volume~69 of {\em Math\'{e}matiques \& Applications}.
\newblock Springer Berlin, Heidelberg, 2012.

\bibitem{MR3045658}
R.~S. Falk and M.~Neilan.
\newblock Stokes complexes and the construction of stable finite elements with
  pointwise mass conservation.
\newblock {\em SIAM J. Numer. Anal.}, 51(2):1308--1326, 2013.

\bibitem{MR851383}
V.~Girault and P.-A. Raviart.
\newblock {\em Finite Element Methods for {N}avier-{S}tokes Equations},
  volume~5 of {\em Springer Series in Computational Mathematics}.
\newblock Springer-Verlag, Berlin, 1986.
\newblock Theory and algorithms.

\bibitem{MR3269433}
J.~Guzm\'{a}n and M.~Neilan.
\newblock Conforming and divergence-free {S}tokes elements in three dimensions.
\newblock {\em IMA J. Numer. Anal.}, 34(4):1489--1508, 2014.

\bibitem{MR3120580}
J.~Guzm\'{a}n and M.~Neilan.
\newblock Conforming and divergence-free {S}tokes elements on general
  triangular meshes.
\newblock {\em Math. Comp.}, 83(285):15--36, 2014.

\bibitem{MR3882274}
J.~Guzm\'{a}n and L.~R. Scott.
\newblock The {S}cott-{V}ogelius finite elements revisited.
\newblock {\em Math. Comp.}, 88(316):515--529, 2019.

\bibitem{MR3712173}
J.~Guzm\'{a}n, C.-W. Shu, and F.~A. Sequeira.
\newblock {$\rm H(div)$} conforming and {DG} methods for incompressible
  {E}uler's equations.
\newblock {\em IMA J. Numer. Anal.}, 37(4):1733--1771, 2017.

\bibitem{MR4200737}
Y.~B. Han and Y.~R. Hou.
\newblock Robust error analysis of {H}(div)-conforming {DG} method for the
  time-dependent incompressible {N}avier-{S}tokes equations.
\newblock {\em J. Comput. Appl. Math.}, 390:Paper No. 113365, 13, 2021.

\bibitem{MR4410752}
Y.~B. Han and Y.~R. Hou.
\newblock Semirobust analysis of an {$\rm H(div)$}-conforming {DG} method with
  semi-implicit time-marching for the evolutionary incompressible
  {N}avier-{S}tokes equations.
\newblock {\em IMA J. Numer. Anal.}, 42(2):1568--1597, 2022.

\bibitem{MR3043640}
F.~Hecht.
\newblock New development in freefem++.
\newblock {\em J. Numer. Math.}, 20(3-4):251--265, 2012.

\bibitem{Hillewaert:2013-Development}
K.~Hillewaert.
\newblock {\em Development of the Discontinuous Galerkin Method for
  High-Resolution, Large Scale CFD and Acoustics in Industrial Geometries}.
\newblock Ph{D} thesis, Université catholique de Louvain, 2013.

\bibitem{MR571681}
T.~J.~R. Hughes and A.~Brooks.
\newblock A multidimensional upwind scheme with no crosswind diffusion.
\newblock In {\em Finite element methods for convection dominated flows
  ({P}apers, {W}inter {A}nn. {M}eeting {A}mer. {S}oc. {M}ech. {E}ngrs., {N}ew
  {Y}ork, 1979)}, pages 19--35. Amer. Soc. Mech. Engrs. (ASME), New York, 1979.

\bibitem{MR3561143}
V.~John.
\newblock {\em Finite Element Methods for Incompressible Flow Problems},
  volume~51 of {\em Springer Series in Computational Mathematics}.
\newblock Springer, Cham, 2016.

\bibitem{MR3683678}
V.~John, A.~Linke, C.~Merdon, M.~Neilan, and L.~G. Rebholz.
\newblock On the divergence constraint in mixed finite element methods for
  incompressible flows.
\newblock {\em SIAM Rev.}, 59(3):492--544, 2017.

\bibitem{MR1485996}
V.~John, J.~M. Maubach, and L.~Tobiska.
\newblock Nonconforming streamline-diffusion-finite-element-methods for
  convection-diffusion problems.
\newblock {\em Numer. Math.}, 78(2):165--188, 1997.

\bibitem{MR2657851}
G.~Kanschat and B.~Rivi\`ere.
\newblock A strongly conservative finite element method for the coupling of
  {S}tokes and {D}arcy flow.
\newblock {\em J. Comput. Phys.}, 229(17):5933--5943, 2010.

\bibitem{MR2860674}
J.~K\"{o}nn\"{o} and R.~Stenberg.
\newblock {$H ({\rm div})$}-conforming finite elements for the {B}rinkman
  problem.
\newblock {\em Math. Models Methods Appl. Sci.}, 21(11):2227--2248, 2011.

\bibitem{MR3133522}
A.~Linke.
\newblock On the role of the {H}elmholtz decomposition in mixed methods for
  incompressible flows and a new variational crime.
\newblock {\em Comput. Methods Appl. Mech. Engrg.}, 268:782--800, 2014.

\bibitem{MR3481034}
A.~Linke and C.~Merdon.
\newblock On velocity errors due to irrotational forces in the
  {N}avier-{S}tokes momentum balance.
\newblock {\em J. Comput. Phys.}, 313:654--661, 2016.

\bibitem{MR483555}
P.-A. Raviart and J.~M. Thomas.
\newblock A mixed finite element method for 2nd order elliptic problems.
\newblock In I.~Galligani and E.~Magenes, editors, {\em Mathematical aspects of
  finite element methods}, Lecture Notes in Mathematics, Vol. 606, pages
  292--315. Springer, Berlin, Heidelberg, 1977.

\bibitem{MR431752}
P.-A. Raviart and J.~M. Thomas.
\newblock Primal hybrid finite element methods for {$2$}nd order elliptic
  equations.
\newblock {\em Math. Comp.}, 31(138):391--413, 1977.

\bibitem{MR3780790}
P.~W. Schroeder and G.~Lube.
\newblock Divergence-free {$H({\rm div})$}-{FEM} for time-dependent
  incompressible flows with applications to high {R}eynolds number vortex
  dynamics.
\newblock {\em J. Sci. Comput.}, 75(2):830--858, 2018.

\bibitem{MR818790}
L.~R. Scott and M.~Vogelius.
\newblock Conforming finite element methods for incompressible and nearly
  incompressible continua.
\newblock In {\em Large-scale computations in fluid mechanics, {P}art 2 ({L}a
  {J}olla, {C}alif., 1983)}, volume~22 of {\em Lectures in Appl. Math.}, pages
  221--244. Amer. Math. Soc., Providence, RI, 1985.

\bibitem{MR813691}
L.~R. Scott and M.~Vogelius.
\newblock Norm estimates for a maximal right inverse of the divergence operator
  in spaces of piecewise polynomials.
\newblock {\em RAIRO Mod\'{e}l. Math. Anal. Num\'{e}r.}, 19(1):111--143, 1985.

\bibitem{MR725982}
R.~Stenberg.
\newblock Analysis of mixed finite elements methods for the {S}tokes problem: a
  unified approach.
\newblock {\em Math. Comp.}, 42(165):9--23, 1984.

\bibitem{MR2594344}
R.~Stenberg.
\newblock A nonstandard mixed finite element family.
\newblock {\em Numer. Math.}, 115(1):131--139, 2010.

\bibitem{MR2249024}
V.~Thom\'{e}e.
\newblock {\em Galerkin Finite Element Methods for Parabolic Problems},
  volume~25 of {\em Springer Series in Computational Mathematics}.
\newblock Springer-Verlag, Berlin, second edition, 2006.

\bibitem{MR2772092}
S.~Y. Zhang.
\newblock Divergence-free finite elements on tetrahedral grids for {$k\geq 6$}.
\newblock {\em Math. Comp.}, 80(274):669--695, 2011.

\end{thebibliography}








%
%

\end{document}